\documentclass[12pt]{article}
\usepackage{amssymb}
\usepackage{amsfonts}
\usepackage{amsmath}
 \usepackage{color}
\usepackage{mathrsfs}
\usepackage{amsfonts}
\usepackage{amssymb,amsmath}
\usepackage{CJK}
\usepackage{cite}
\usepackage{cases}
\usepackage{amsthm}

\def\cl{\centerline}

\def\vs{\vspace*}
\def\W{\mathcal{L}}

\def\Z{\mathbb{Z}}

\def\C{\mathbb{C}}

\numberwithin{equation}{section}
\newtheorem{theo}{Theorem}[section]
\newtheorem{defi}[theo]{Definition}
\newtheorem{coro}[theo]{Corollary}
\newtheorem{lemm}[theo]{Lemma}
\newtheorem{exam}[theo]{Example}
\newtheorem{prop}[theo]{Proposition}

\newtheorem{case}{Case}

\newtheorem{rema}[theo]{Remark}

\begin{document}
\begin{center}
{\bf\large Finite irreducible conformal modules over the Lie conformal superalgebra $\mathcal{S}(p)$}
\footnote {

$^*$Corresponding author: Haibo Chen (rebel1025@126.com).
}
\end{center}

\cl{Haibo Chen$^*$, Yanyong Hong and    Yucai Su
}

\vs{8pt}

{\small\footnotesize
\parskip .005 truein
\baselineskip 3pt \lineskip 3pt
\noindent{{\bf Abstract:} In the present paper,  we introduce a class of infinite Lie conformal superalgebras $\mathcal{S}(p)$, which are closely related to  Lie conformal algebras of extended Block type defined in \cite{CHS}.
Then all finite non-trivial irreducible conformal modules over  $\mathcal{S}(p)$ for $p\in\C^*$ are  completely classified.
As an application,  we also present the classifications of finite non-trivial irreducible conformal modules over finite quotient algebras $\mathfrak{s}(n)$ for $n\geq1$ and $\mathfrak{sh}$ which is isomorphic to  a subalgebra of Lie conformal algebra of $N=2$ superconformal algebra.
Moreover, as a generalized version of $\mathcal{S}(p)$, the infinite Lie conformal superalgebras $\mathcal{GS}(p)$ are constructed,  which have a
  subalgebra   isomorphic to the finite Lie conformal algebra of $N=2$ superconformal algebra.
  \vs{5pt}

\parskip .001 truein\baselineskip 6pt \lineskip 6pt

}}
\parskip .001 truein\baselineskip 6pt \lineskip 6pt
\tableofcontents

\section{Introduction}
The notion of Lie conformal  (super)algebras was originally  introduced  by Kac in \cite{K1,K3},
which encodes an axiomatic description of the
operator product expansion (or rather its Fourier transform) of chiral fields in conformal field theory.
 It plays important
roles in quantum field theory, vertex algebras, integrable systems  and so on, and  has drawn many researchers'  extensive attentions.
As is well known, the theory of Lie conformal (super)algebras gives us a powerful tool
for the study of infinite-dimensional Lie (super)algebras satisfying
the locality property described in \cite{K2}.


 There are a lot of researches on the case of Lie conformal algebra. It follows from \cite{DK}  that  Virasoro Lie conformal algebra and all current
Lie conformal algebras with finite-dimensional simple Lie algebras  exhaust all
finite simple Lie conformal algebras.
The  theory of  finite Lie conformal algebras associated with Virasoro Lie conformal algebra were intensively studied
(see, e.g., \cite{BKV, CK, CK1, DK, WY, LHW}).
Furthermore, the problem of classifying finite simple Lie conformal superalgebras was completely solved
in \cite{FK}.  It  shows that any finite simple Lie conformal superalgebra   is isomorphic to
one of the Lie conformal superalgebras of the  list, which includes current Lie conformal superalgebras over finite-dimensional simple Lie superalgebras,    four series of ``Virasoro-like'' Lie conformal
superalgebras and
exceptional Lie conformal superalgebra $CK_6$.
It is  worth noting that  finite  non-trivial irreducible conformal modules of them  were completely classified in \cite{CK,MZ,CL,BKL1,BKL2,BKL3}.

The infinite Lie conformal  (super)algebra has become one of the major research objects in conformal algebraic theory.
In order to better understand the theory of  infinite Lie conformal (super)algebras, it is very natural to investigate some important examples.  Based on  the relation of Lie algebras and Lie conformal algebras, some   infinite  Lie conformal algebras   were defined   by
 loop Lie algebras and Block type Lie algebras  (see, e.g.,  \cite{WCY,CHSX,FSW,HW,CHS,SXY,SY2}). In a similar way,  a class of infinite Lie conformal superalgebras called loop Virasoro conformal superalgebras were constructed in \cite{DH}, which are associated with the
loop super-Virasoro algebras.
Recently, the infinite Lie conformal superalgebras  of Block type were introduced in \cite{X},
which  contain a  Neveu-Schwarz conformal  subalgebra.
 At the same times, their finite non-trivial irreducible conformal modules   were classified for   $p\neq0$.
  But, for all we know, there are very few works about the  infinite Lie conformal superalgebras.

In this paper, we  define  a new  class of infinite  Lie conformal superalgebras $\mathcal{S}(p)$ with $p\neq0$,
 which  is related to  a class of  extended Block type Lie conformal algebras $\mathfrak{B}(\alpha,\beta,p)$ studied in \cite{CHS}.
{\it The Lie conformal superalgebras of extended  Block type  are $\mathcal{S}(p)=\mathcal{S}(p)_{\bar0}\oplus \mathcal{S}(p)_{\bar1}$}
 with
 $\mathcal{S}(p)_{\bar0}=\oplus_{i\in\Z_+}\C[\partial]L_i\bigoplus \oplus_{i\in\Z_+}\C[\partial]W_i$,
 $\mathcal{S}(p)_{\bar1}=\oplus_{i\in\Z_+}\C[\partial]G_i$   and   $\lambda$-brackets as follows
\begin{eqnarray}
&&\label{11} [L_i\, {}_\lambda \, L_j]=\big((i+p)\partial+(i+j+2p)\lambda \big) L_{i+j},
\\&&[L_i\, {}_\lambda \, W_j]=\Big((i+p)\partial+(i+j+p)\lambda\Big) W_{i+j},
\\&&[L_i\, {}_\lambda \, G_j]=\Big((i+p) \partial +(i+j+2p)\lambda\Big) G_{i+j},
\\&&\label{14}
[W_i\, {}_\lambda \, G_j]=G_{i+j},\ [W_i\, {}_\lambda \, W_j]=[G_i\, {}_\lambda \, G_j]=0
\end{eqnarray}
for $i,j\in\Z_+$. Note that  the even part $\mathcal{S}(p)_{\bar0}$ of $\mathcal{S}(p)$ is
an extended Block type Lie conformal algebra $\mathfrak{B}(\alpha,\beta,p)$ for $\alpha=p,\beta=0$.
 The subalgebra  $\mathfrak{H}=\C[\partial](\frac{1}{p}L_0)\oplus  \C[\partial]W_0$  of $\mathcal{S}(p)_{\bar0}$
 is  so-called  Heisenberg-Virasoro Lie conformal algebra.

This article is organized as follows.

 In Section $2$, we introduce some
basic definitions  and related known results about Lie conformal superalgebras and conformal modules.

In Section 3, by  recalling the   definition of  $\mathfrak{B}(\alpha,\beta,p)$ and certain module structures, we define a class of Lie conformal superalgebras $\mathcal{S}(p)$. Then we investigate  their  subalgebras, quotient algebras and representations  of annihilation superalgebras.

In Section 4,   the irreducibility of all free
non-trivial rank $1+1$ modules over $\mathcal{S}(p)$ are determined. A complete classification
of all finite non-trivial irreducible conformal modules of $\mathcal{S}(p)$ are given, which shows that they must be free of rank $1$ or $1+1$.

In Section $5$, we construct a class of    Lie conformal superalgebras, which are generalizations  of Lie conformal superalgebras  $\mathcal{S}(p)$.   They have some   subalgebras, one of them is exactly the Lie conformal algebra of $N=2$ superconformal algebra.

At last, as a  byproduct  of our main result, we also obtain the classification of all finite non-trivial irreducible conformal modules  over the subalgebra $\mathfrak{sh}$ and quotient algebras $\mathfrak{s}(n)$ for $n\geq1$.

Throughout this paper, all vector spaces, linear maps and tensor products  are  assumed to be  over
complex field $\C$. We denote by  $\C^*$, $\Z$ and $\Z_+$ the sets of   nonzero complex numbers,  integers and nonnegative integers, respectively. Moreover, if $A$ is a vector space, the space of polynomials of $\lambda$ with coefficients in $A$ is denoted by $A[\lambda]$.

\section{Preliminaries}
In this section, we recall some basic concepts and results related to Lie conformal superalgebras and conformal
modules in \cite{DK,K1,K3}.

We denote $\Z_2=\{\bar0,\bar1\}$. A vector space $U$ is called $\Z_2$-graded if $U=U_{\bar 0}\oplus U_{\bar1}$, and $u\in U_{\bar i}$ is called
$\Z_2$-homogenous and write $|u|=\bar i$.

\begin{defi}\label{D1}
A   Lie  conformal  superalgebra  $S=S_{\bar0}\oplus S_{\bar1}$ is a $\Z_2$-graded $\C[\partial]$-module  endowed with a
  $\lambda$-bracket $[a{}\, _\lambda \, b]$
which defines a
linear map $S_\alpha\otimes S_\beta\rightarrow \C[\lambda]\otimes S_{\alpha+\beta}$, where $\lambda$ is an indeterminate, and satisfy the following axioms:
\begin{equation*}
\aligned
&[\partial a\,{}_\lambda \,b]=-\lambda[a\,{}_\lambda\, b],\
[a\,{}_\lambda \,\partial b]=(\partial+\lambda)[a\,{}_\lambda\, b],\\
&[a\, {}_\lambda\, b]=-(-1)^{|a||b|}[b\,{}_{-\lambda-\partial}\,a],\\
&[a\,{}_\lambda\,[b\,{}_\mu\, c]]=[[a\,{}_\lambda\, b]\,{}_{\lambda+\mu}\, c]+(-1)^{|a||b|}[b\,{}_\mu\,[a\,{}_\lambda \,c]]
\endaligned
\end{equation*}
for all $c\in S$, $\Z_2$-homogenous elements $a$, $b$ in $S$, and $\alpha$, $\beta\in \Z_2$.
\end{defi}

A Lie conformal superalgebra is called {\it finite} if it is finitely generated as a $\C[\partial]$-module, or else it is called {\it infinite}.

\begin{defi}\label{defi-module} A  conformal module $M=M_{\bar0}\oplus M_{\bar1}$ over a Lie conformal
superalgebra $S$ is a  $\Z_2$-graded $\C[\partial]$-module endowed with a $\lambda$-action $S_\alpha\otimes M_\beta\rightarrow \C[\lambda]\otimes M_{\alpha+\beta}$ such that
\begin{eqnarray*}
&&(\partial a)\,{}_\lambda\, v=-\lambda a\,{}_\lambda\, v,\ a{}\,{}_\lambda\, (\partial v)=(\partial+\lambda)a\,{}_\lambda\, v,
\\&&
a\,{}_\lambda\, (b{}\,_\mu\, v)-(-1)^{|a||b|}b\,{}_\mu\,(a\,{}_\lambda\, v)=[a\,{}_\lambda\, b]\,{}_{\lambda+\mu}\, v
\end{eqnarray*}
for all $v\in M$, $\Z_2$-homogenous elements $a$, $b$ in $S$, and $\alpha$, $\beta\in \Z_2$.
\end{defi}

Let $M=M_{\bar0}\oplus M_{\bar1}$ be a conformal $S$-module. Obviously, there is a parity-change functor $\Pi$ from the category of $S$-modules to itself, which implies  that  a new module $\Pi(M)$ is obtained by
  $\Pi(M_{\bar0})=M_{\bar1}$ and $\Pi(M_{\bar1})=M_{\bar0}$.
   The  module    $M$ is called {\it finite} if it is finitely generated over $\C[\partial]$.
 We call  that the {\it rank} of       $M$  is $m+n$ as a $\C[\partial]$-module, if the rank of $M_{\bar0}$ is $m$ and  the rank of $M_{\bar1}$ is $n$.  If    $M$  has no non-trivial submodules, the conformal module $M$ is
called {\it irreducible}.

 A Lie conformal superalgebra $S$ is called $\Z$-graded if $S=\oplus_{i\in\Z}S_i$, each  $S_i$ is a $\C[\partial]$-submodule and
 $[S_i\,{}_\lambda\, S_j]\subseteq S_{i+j}[\lambda]$  for any $i, j\in\Z$.
  The conformal module $M$   is $\Z$-graded if $M=\oplus_{i\in\Z}M_i$, each $M_i$ is a $\C[\partial]$-submodule and $S_i\,{}_\lambda\, M_j\subseteq M_{i+j}[\lambda]$ for any $i,j\in\Z$.
Furthermore, if each $M_i$ is freely generated by an element $v_i\in M_i$
over $\C[\partial]$, then $M$ is called a {\it  $\Z$-graded free intermediate series module}.

\begin{defi}\label{d2.3}
 An annihilation superalgebra $\mathcal{A}(S)$ of a Lie conformal superalgebra $S$ is a Lie superalgebra with $\C$-basis $\{a(n)\mid a\in S,n\in\Z_+\}$ and relations $\mathrm{(}$for any $a$, $b\in S$ and $k\in \C$$\mathrm{)}$
\begin{eqnarray}
 \label{2asd2.1} &&(ka)_{(n)}=ka_{(n)},\ (a+b)_{(n)}=a_{(n)}+b_{(n)},\\
&&\label{22.1}[a_{(m)},b_{(n)}]=\sum_{k\in\Z_+}{m\choose k}(a_{(k)}b)_{(m+n-k)}, \ (\partial a)_{(n)}=-na_{(n-1)},
\end{eqnarray}
where $a(n)\in \mathcal{A}(S)_\alpha$ if $a\in S_\alpha$, and $a_{(k)}b$  is called the $k$-th product, given by $[a\, {}_\lambda \, b]=\sum_{k\in\Z_+}\frac{\lambda^{k}}{k!}(a_{(k)}b) $. Furthermore, an extended annihilation superalgebra $\mathcal{A}(S)^e$ of $S$
 is defined by $\mathcal{A}(S)^e=\C \partial\ltimes \mathcal{A}(S)$  with   $[\partial,a_{(n)}]=-na_{n-1}$, where $\C \partial \subseteq \mathcal{A}(S)^e_{\overline{0}}$.
\end{defi}
Now we can
  define {\it $k$-th actions} of $S$ on $M$ for each $j\in \Z_+$, i.e. $a_{(k)}v$ for any $a\in S, v\in M$
\begin{eqnarray}\label{22.111}
a\, {}_\lambda \, v=\sum_{k\in\Z_+}\frac{\lambda^{(k)}}{k!}(a_{(k)}v),
\end{eqnarray}
which is similar to the  definition of   $k$-th product $a_{(k)}b$ for $a, b \in S$.

The following result appeared in  \cite{CK}, which implies that  a close connection between the module of a Lie
conformal superalgebra and that of its extended annihilation superalgebra.
\begin{prop}\label{pro2.4}
A conformal module $M$ over a Lie conformal superalgebra $S$ is the same as a module over the
 Lie superalgebra $\mathcal{A}(S)^e$  satisfying $a_{(n)}v=0$ for $a\in S,v\in M,n\gg 0$.
\end{prop}

\section{ Lie conformal superalgebra $\mathcal{S}(p)$}

In this section,  a class of  extended Block type  Lie conformal superalgebras $\mathcal{S}(p)$
are defined.

We first recall the definition of the extended Block type Lie conformal algebra  $\mathfrak{B}(\alpha,\beta,p)$ and its  $\Z$-graded intermediate series modules (see \cite{X,CHS}).

  The  infinite Lie conformal algebra called
  {\it extended Block type Lie conformal algebra} $\mathfrak{B}(\alpha,\beta,p)$ with $p\in \C^*$
 has a  $\C[\partial]$-basis $\{L_{i},W_{i}\mid i\in\Z_+ \}$
satisfying   the following non-trivial  $\lambda$-brackets
\begin{eqnarray*}
&& [L_i\, {}_\lambda \, L_j]=\big((i+p)\partial+(i+j+2p)\lambda \big) L_{i+j},
\\&& [L_i\, {}_\lambda \, W_j]=\Big((i+p)(\partial+\beta)+(i+j+\alpha)\lambda\Big) W_{i+j}
\end{eqnarray*}
for any $\alpha,\beta\in\C$. In the following, we only consider $\alpha=p,\beta=0$.
For  $\alpha_1,\beta_1,\gamma_1\in\C,p \in\C^*$,   the   $\C[\partial]$-module $V(\alpha_1,\beta_1,\gamma_1,p)=\bigoplus_{i\in\Z}\C[\partial]v_i$  is a $\Z$-graded free intermediate
series module over $\mathfrak{B}(\alpha,\beta,p)$ with  $\lambda$-actions as follows:
\begin{eqnarray}\label{v111}
 L_i\, {}_\lambda \, v_j=\Big((i+p)(\partial+\beta_1)+(i+j+\alpha_1)\lambda\Big)v_{i+j},
 \ W_i\, {}_\lambda \, v_j=\gamma_1 v_{i+j}.
\end{eqnarray}
 Inspired by this, we consider a $\Z_2$-graded $\C[\partial]$-module
$$\mathcal{S}(\alpha_1,\beta_1,\gamma_1,p)=\mathcal{S}_{\bar 0}\oplus\mathcal{S}_{\bar 1}$$
with
$\mathcal{S}_{\bar 0}=\oplus_{i\in\Z_+}\C[\partial]L_i\bigoplus\oplus_{i\in\Z_+}\C[\partial]W_i,
\mathcal{S}_{\bar 1}=\oplus_{i\in\Z_+}\C[\partial]G_i$, and satisfying
\begin{eqnarray}
&&\nonumber  [L_i\, {}_\lambda \, G_j]=\Big((i+p)(\partial+\beta_1)+(i+j+\alpha_1)\lambda\Big) G_{i+j},
\\&&\nonumber  [W_i\, {}_\lambda \, G_j]=\gamma_1  G_{i+j},
\ [G_i\, {}_\lambda \, G_j]=0
\end{eqnarray}
for $i,j\in\Z_+$. Let $\alpha_1=2p$, $\beta_1=0$ and $\gamma_1=1$.  Then the  $\Z_2$-graded $\C[\partial]$-module  $\mathcal{S}(\alpha_1,\beta_1, \gamma_1,p)$ becomes the Lie conformal superalgebra $\mathcal{S}(p)$, which  is exactly what we defined in \eqref{11}-\eqref{14}.

Now we present some interesting features on  $\mathcal{S}(p)$ as follows.
\subsection{Subalgebras}
Setting $L=\frac{1}{p}L_0,W= W_0,G=G_0\in \mathcal{S}(p)$ in  \eqref{11}-\eqref{14},  we can obtain the non-vanishing relations as follows:
$$
[L\, {}_\lambda \, L]=(\partial+2\lambda)L,\
  [L\, {}_\lambda \, W]=(\partial+\lambda)W,
\ [L\, {}_\lambda \, G]=(\partial+2\lambda)G, \ [W\, {}_\lambda \, G]=G,
$$
which  is called  Heisenberg-Virasoro Lie conformal superalgebra $\mathfrak{sh}$.
Clearly, the  even part  of $\mathfrak{sh}$ is  Heisenberg-Virasoro Lie  conformal algebra, which has been studied extensively   (see, e.g., \cite{WY,LHW,CHS}).
 We see that $\C[\partial](L+aW)$ for $a\in\C$ spans a  subalgebra of Heisenberg-Virasoro Lie conformal algebra  which is isomorphic to the   Virasoro Lie conformal algebra.
 Now we define the following $\C[\partial]$-module homomorphism from $\mathfrak{sh}$ to Lie conformal algebra of $N=2$ superconformal algebra (see \cite{CL}):
\begin{eqnarray}\label{ln23}
L+\frac{1}{2}\partial W \rightarrow L,\ W\rightarrow J,\ G\rightarrow G^+.
\end{eqnarray}
Then it is easy to check that  $\mathfrak{sh}$ is isomorphic to a subalgebra of Lie conformal algebra of
$N=2$ superconformal algebra.

\subsection{Quotient algebras}
Many finite Lie conformal superalgebras will be obtained by considering the quotient algebras of $\mathcal{S}(p)$.
We note that $\mathcal{S}(p)$ is $\Z$-graded in  the sense of
\begin{eqnarray}\label{b3547877wq72}
\mathcal{S}(p)=\oplus_{k\in\Z_+}\mathcal{S}(p)_k,
\end{eqnarray}
 where $\mathcal{S}(p)_k=\C[\partial] L_k\oplus\C [\partial] W_k\oplus\C [\partial] G_k$. For any $n\in\Z_+,$ we can define a subspace $\mathcal{S}(p)_{\langle n\rangle}$
of $\mathcal{S}(p)$ by
$$\mathcal{S}(p)_{\langle n\rangle}=\oplus_{i\geq n}\C[\partial]L_i\bigoplus\oplus_{i\geq n}\C[\partial]W_i\bigoplus\oplus_{i\geq n}\C[\partial]G_i.$$
Obviously,   $\mathcal{S}(p)_{\langle n\rangle}$ is an ideal of  the  Lie conformal superalgebra of $\mathcal{S}(p)$.
For  $n\in\Z_+$,  define
\begin{eqnarray}\label{b3.2}
\mathcal{S}(p)_{[n]}=\mathcal{S}(p)/\mathcal{S}(p)_{\langle n+1\rangle}.
\end{eqnarray}
Note that $\mathcal{S}(p)_{[0]}\cong \mathfrak{sh}$. Choosing $p=-n$ for $1\leq n\in\Z$,  we can define  the quotient algebras
   $\mathcal{S}(-n)_{[n]}$
 by the following relations
\begin{eqnarray}\label{bn3.2}
\mathfrak{s}(n)=\mathcal{S}(-n)_{[n]}=\mathcal{S}(-n)/\mathcal{S}(-n)_{\langle n+1\rangle}.
\end{eqnarray}
 Then a series of new finite non-simple Lie conformal superalgebras can be produced.
Now we  give the following two examples for $n=1,2$.
\begin{exam}
Setting  $L=-\bar L_0,W=\bar W_0,G=\bar G_0,M=\bar L_1,H=\bar W_1,I=\bar G_1\in \mathfrak{s}(1)$, one can obtain the following non-trivial relations
 \begin{eqnarray*}
&&[L\, {}_\lambda \, L]=(\partial+2\lambda)L,\
  [L\, {}_\lambda \, W]=(\partial+\lambda)W,
\ [L\, {}_\lambda \, G]=(\partial+2\lambda)G,
\\&&
[W\, {}_\lambda \, G]=G,
\ [L\, {}_\lambda \, M]=(\partial+\lambda)M, \ [L\, {}_\lambda \, H]=\partial H,
\\&& [L\, {}_\lambda \, I]=(\partial+\lambda)I,\
[W\, {}_\lambda \, I]=[H\, {}_\lambda \, G]=I,\ [M\, {}_\lambda \, G]=-\lambda I.
\end{eqnarray*}
Other  $\lambda$-brackets   are given by skew-symmetry. We  observe that  $\C[\partial]L\oplus \C[\partial]W$  and $\C[\partial]L\oplus \C[\partial]M$ are both    Heisenberg-Virasoro Lie conformal algebra. Maybe  $\mathfrak{s}(1)$ should be called      BiHeisenberg-Virasoro Lie conformal superalgebra.
\end{exam}

\begin{exam}
Set  $L=-\frac{1}{2}\bar L_0,W=\bar W_0,G=\bar G_0,M=\bar L_1,H=\bar W_1,I=\bar G_1,X=-\bar L_2,Y=-\bar W_2,Z=-\bar G_2\in \mathfrak{s}(2)$.   Then the non-vanishing relations are presented as follows
 \begin{eqnarray*}
&&[L\, {}_\lambda \, L]=(\partial+2\lambda)L,\
  [L\, {}_\lambda \, W]=(\partial+\lambda)W,
\ [L\, {}_\lambda \, G]=(\partial+2\lambda)G, \\&&
 [W\, {}_\lambda \, G]=G,
\ [L\, {}_\lambda \, M]=(\partial+\frac{3}{2}\lambda)M, \
 [L\, {}_\lambda \, H]=(\partial+\frac{1}{2}\lambda) H,
 \\&&
  [L\, {}_\lambda \, I]=(\partial+\frac{3}{2}\lambda)I,\
[W\, {}_\lambda \, I]=I,
\ [L\, {}_\lambda \, X]=(\partial+\lambda)X,\
  [L\, {}_\lambda \, Y]=\partial Y,
\\&&
 [L\, {}_\lambda \, Z]=(\partial+\lambda)Z, \ [W\, {}_\lambda \, Z]=Z,
\ [M\, {}_\lambda \, M]=(\partial+2\lambda)X,\
  [M\, {}_\lambda \, H]=\partial Y,
\\&& [M\, {}_\lambda \, I]=(\partial+2\lambda)Z, \ [H\, {}_\lambda \, I]=Z,
\ [M\, {}_\lambda \, W]=-(\partial+\lambda) H,
\\&&
 [M\, {}_\lambda \, G]=-(\partial+3\lambda)I, \
  [H\, {}_\lambda \, G]=I,
\
[X\, {}_\lambda \, G]=-2\lambda Z,
\ [Y\, {}_\lambda \, G]=Z.
\end{eqnarray*}
Other  $\lambda$-brackets  can be    obtained by skew-symmetry.  Note that  $\C[\partial]L\oplus \C[\partial]M\oplus \C[\partial]X$  and $\C[\partial]L\oplus \C[\partial]W$ are respectively    Schr\"{o}dinger-Virasoro Lie conformal algebra and
   Heisenberg-Virasoro Lie conformal algebra. Maybe  $\mathfrak{s}(2)$ should be called       Schr\"{o}dinger-Heisenberg-Virasoro Lie conformal superalgebra.
\end{exam}

\subsection{Representation   of annihilation superalgebras}
In this section,  the irreducible modules over a subquotient algebra of the
annihilation superalgebra  $\mathcal{A}(\mathcal{S}(p))$ of $\mathcal{S}(p)$ are classified.

Firstly, we provide the explicit super-brackets
of $\mathcal{A}(\mathcal{S}(p))$ as  follows.
\begin{lemm}\label{4.1}
 \begin{itemize}\parskip-7pt
 \item[{\rm (1)}] The annihilation  superalgebra  of $\mathcal{S}(p)$ is
$$\mathcal{A}(\mathcal{S}(p))=\Big\{L_{i,m},W_{j,n},G_{k,l}\mid i,j,k,n\in \Z_+,m,l\in\Z_+\cup\{-1\}\Big\}$$
with the following      Lie super-brackets:
\begin{equation}\label{44.1}
\aligned
&[L_{i,m},L_{j,n}]=\big((m+1)(j+p)-(n+1)(i+p)\big)L_{i+j,m+n},
\\&[L_{i,m},W_{j,n}]=\big((m+1)j-n(i+p)\big)W_{i+j,m+n},
\\&[L_{i,m},G_{j,n}]=\big((m+1)(j+p)-(n+1)(i+p)\big)G_{i+j,m+n},
\\&  [W_{i,m},G_{j,n}]=G_{i+j,m+n},
\ [W_{i,m},W_{j,n}]=[G_{i,m},G_{j,n}]=0,
\endaligned
\end{equation}where $p\in\C^*\mathrm{;}$
 \item[{\rm (2)}] The extended annihilation algebra  is
 $$\mathcal{A}(\mathcal{S}(p))^e=\Big\{L_{i,m},W_{j,n},G_{k,l},\partial\mid i,j,k,n\in\Z_+,m,l\in \Z_+\cup\{-1\}\Big\}$$
satisfying
\eqref{44.1} and
$$[\partial,L_{i,m}]=-(m+1)L_{i,m-1},\ [\partial,W_{j,n}]=-nW_{j,n-1},\ [\partial,G_{k,l}]=-(l+1)G_{k,l-1}.$$
\end{itemize}

\end{lemm}

\begin{proof}
By the definition of the $k$-th product    in Definition  \ref{d2.3} and   $\mathcal{S}(p)$,    we conclude that
\begin{eqnarray*}
L_i\,{}_{{}_{(k)}} L_j&=&
\begin{cases}
(i+p)\partial L_{i+j} &\ \mbox{if}\  k=0,\\[4pt]
(i+j+2p)L_{i+j}&\  \mbox{if} \ k=1,\\[4pt]
0&\  \mbox{if} \ k\geq2,
\end{cases}\\
L_i\,{}_{{}_{(k)}} W_j&=&
\begin{cases}
(i+p)\partial W_{i+j} &\ \mbox{if}\  k=0,\\[4pt]
(i+j+p)  W_{i+j}&\  \mbox{if} \ k=1,\\[4pt]
0&\  \mbox{if} \ k\geq2,
\end{cases}\\
L_i\,{}_{{}_{(k)}} G_j&=&
\begin{cases}
(i+p)\partial G_{i+j} &\ \mbox{if}\  k=0,\\[4pt]
(i+j+2p)  G_{i+j}&\  \mbox{if} \ k=1,\\[4pt]
0&\  \mbox{if} \ k\geq2,
\end{cases}\\
W_i\,{}_{{}_{(k)}} G_j&=&
\begin{cases}
  G_{i+j} &\ \mbox{if}\  k=0,\\[4pt]
0&\  \mbox{if} \ k\geq1,
\end{cases}\\
G_i\,{}_{{}_{(k)}} G_j&=&
W_i\,{}_{{}_{(k)}} W_j=0 \quad \mathrm{for \ any}\ k\in\Z_+.
\end{eqnarray*}
From  \eqref{2asd2.1} and \eqref{22.1}, it is easy to see that
 \begin{equation}\label{44.2}
\aligned
&[(L_i)_{(m)},(L_j)_{(n)}]=\big(m(j+p)-n(i+p)\big)(L_{i+j})_{(m+n-1)},
\\&[(L_i)_{(m)},(W_j)_{(n)}]=\big(mj-n (i+p)\big)(W_{i+j})_{(m+n-1)},
\\&[(L_i)_{(m)},(G_j)_{(n)}]=\big(m(j+p)-n (i+p)\big)(G_{i+j})_{(m+n-1)},
\\&[(W_i)_{(m)},(G_j)_{(n)}]=(G_{i+j})_{(m+n)},
\ [(G_i)_{(m)},(G_j)_{(n)}]=0,
\\&[(W_i)_{(m)},(W_j)_{(n)}]=0,
\  [\partial,(L_i)_{(m)}]=-m(L_i)_{(m-1)},
\\&  [\partial,(W_j)_{(n)}]=-n(W_j)_{(n-1)},
\  [\partial,(G_k)_{(l)}]=-l(G_k)_{(l-1)}.
\endaligned
\end{equation}
Then the lemma is proved by  setting $L_{i,m}=(L_i)_{(m+1)}, W_{j,n}=(W_j)_{(n)}$ and $G_{k,l}=(G_k)_{(l+1)}$ in \eqref{44.2} for $i,j,k,n\in\Z_+,m,l\in\Z_+\cup\{-1\}$.
\end{proof}
\begin{rema}
  It  has come to our notice that
  the super Heisenberg-Virasoro algebra  is isomorphic
to the Lie superalgebra   generated by $\{L_{0,m},W_{0,n},G_{0,l}\mid m,n,l\in\Z\}$ in  $\mathcal{A}(\mathcal{S}(p))$ $\mathrm{(}$see \cite{M}$\mathrm{)}$.
\end{rema}

Secondly, we  investigate the representation theory of a subquotient algebra of $\mathcal{A}(\mathcal{S}(p))$. It is obvious that
$$\mathcal{A}(\mathcal{S}(p))^+=\mathcal{A}(\mathcal{S}(p))^+_{\bar0}\oplus
\mathcal{A}(\mathcal{S}(p))^+_{\bar1}$$
is a subalgebra of $\mathcal{A}(\mathcal{S}(p))$, where $\mathcal{A}(\mathcal{S}(p))^+_{\bar0}=\{L_{i,m},W_{j,n}\mid i,j,m,n\in \Z_+\}$
and $\mathcal{A}(\mathcal{S}(p))^+_{\bar1}=\{G_{k,l}\mid k,l\in \Z_+\}$.
For any   $t,N\in\Z_+$, we denote
$$\mathcal{I}(t,N)=\mathcal{I}(k,N)_{\bar0}\oplus\mathcal{I}(k,N)_{\bar1},$$
where $\mathcal{I}(t,N)_{\bar0}=\{L_{i,m},W_{j,n}\in\mathcal{A}(\mathcal{S}(p))_+\mid i,j>t,m,n>N\}$  and $\mathcal{I}(t,N)_{\bar1}=\{G_{k,l}\in\mathcal{A}(\mathcal{S}(p))_+\mid k>t,l>N\}$.
We see that $\mathcal{I}(t,N)$ is an ideal of $\mathcal{A}(\mathcal{S}(p))^+$.
Denote  $$\mathfrak{p}(t,N)=\mathfrak{p}(t,N)_{\bar0}\oplus\mathfrak{p}(t,N)_{\bar1}
=\mathcal{A}(\mathcal{S}(p))^+/\mathcal{I}(t,N).$$
For the later use, denote the following ideals of $\mathfrak{p}(t,N)$ for $t,N\geq1$:
\begin{eqnarray*}
&&\chi(t,N)=\mathrm{span}_{\C}\{ \bar L_{t,m},\bar W_{t,n},\bar G_{t,l}\in\mathfrak{p}(t,N)\mid m,n,l\leq N\},
\\&&\psi(t,N)=\mathrm{span}_{\C}\{\bar L_{i,N},\bar W_{j,N},\bar G_{k,N}\in\mathfrak{p}(t,N)\mid i,j,k\leq t\},
\\&&\phi(t,N)=\mathrm{span}_{\C}\{\bar L_{t,m},\bar W_{t,n},\bar G_{t,l},\bar L_{i,N},\bar W_{j,N},\bar G_{k,N}\in\mathfrak{p}(t,N)\mid m,n,l\leq N,i,j,k\leq t-1\}.
\end{eqnarray*}
We also denote   two finite sets:
$$\Phi=\{(i,m)\mid\bar L_{i,m},\bar W_{i,m},\bar G_{i,m}\in\mathfrak{p}(t,N)\}\backslash\{(0,0)\},
\ \Phi_0=\{(i,m)\in\Phi\mid i-pm=0\}.$$
\begin{lemm}\label{cl2222}
Let $t,N\geq1$, $\Phi_0\neq\emptyset$ and $$i_0=\mathrm{max}\{i\mid(i,m)\in\Phi_0\},\ m_0=\mathrm{max}\{m\mid(i,m)\in\Phi_0\}.$$ Assume that $V=V_{\bar0}\oplus V_{\bar1}$ is a non-trivial finite-dimensional irreducible module over $\mathfrak{p}(t,N)$.
\begin{itemize}\parskip-7pt
 \item[{\rm (1)}]
  If $i_0<t$, then the ideal  $\chi(t,N)$ of $\mathfrak{p}(t,N)$ acts trivially on $V$$\mathrm{;}$
 \item[{\rm (2)}]
  If $m_0<N$, then the ideal  $\psi(t,N)$ of $\mathfrak{p}(t,N)$ acts trivially on $V$$\mathrm{;}$
  \item[{\rm (3)}]
   If $i_0=t,m_0=N$, then the ideal  $\phi(t,N)$ of $\mathfrak{p}(t,N)$ acts trivially on $V$.
\end{itemize}
\end{lemm}
\begin{proof}
(1)   Consider the action of $\bar L_{0,0}$ on $\chi(t,N)$:
$$[\bar L_{0,0},\bar L_{i,m}]=(i-mp)\bar L_{i,m},\ [\bar L_{0,0},\bar W_{i,m}]=(i-mp)\bar W_{i,m},\ [\bar L_{0,0},\bar G_{i,m}]=(i-mp)\bar G_{i,m},$$
where $0\leq i\leq t, 0\leq m\leq N$.
 It follows from $t>i_0$ that $(t-mp)\neq0$. Obviously,     $\chi(t,N)$ is a completely reducible
$\C \bar L_{0,0}$-module with no trivial summand. By  Lemma $1$ of \cite{CK},
we get that  $\chi(t,N)$  acts trivially on $V$.
Similarly,  (2) can be obtained.

(3) Note that $p>0$. Assume that $\phi(t,N)$ acts non-trivially on $V$. According to the irreducibility of $V$, one can see that $V=\phi(t,N)V$. Choose a decomposition of $\phi(t,N)$ as follows
$$\phi(t,N)=\mathrm{span}_{\C} \{\bar L_{i_0,m_0},\bar W_{i_0,m_0},\bar G_{i_0,m_0}\}+\widetilde{\phi}(t,N),$$
where
$\widetilde{\phi}(t,N)=\phi(t,N)\setminus\mathrm{span}_{\C}   \{\bar L_{i_0,m_0},\bar W_{i_0,m_0},\bar G_{i_0,m_0}\}$.

Considering the action of $\bar L_{0,0}$ on $\widetilde{\phi}(t,N)$,  we know that every element in $\widetilde{\phi}(t,N)$ acts nilpotently on $V$  by Lemma 1 of \cite{CK}. From
\begin{eqnarray*}
&&[\bar L_{i_0,0},\bar L_{0,m_0}]=-((i_0+p)m_0+i_0)\bar L_{i_0,m_0},
\\&&[\bar L_{i_0,0},\bar W_{0,m_0}]=-(i_0+p)m_0\bar W_{i_0,m_0},
\\&&[\bar L_{i_0,0},\bar G_{0,m_0}]=-((i_0+p)m_0+i_0)\bar G_{i_0,m_0},
\end{eqnarray*}
we check that  $\mathrm{span}_{\C} \{\bar L_{i_0,m_0},\bar W_{i_0,m_0},\bar G_{i_0,m_0}\}$ acts trivially on $V$. The results hold.
\end{proof}
\begin{lemm}\label{lemm555}
Let $V=V_{\bar0}\oplus V_{\bar1}$ be a non-trivial finite-dimensional irreducible module over $\mathfrak{p}(t,N)$.
Then we obtain $\mathrm{dim}(V)=\mathrm{dim}(V_{\bar0})=1$ or $\mathrm{dim}(V)=1+1$.
\end{lemm}
\begin{proof}
Regard  $V$
 as a finite-dimensional  $\mathfrak{p}(t,N)_{\bar0}$-module. By Lemma 3.5 of \cite{CHS}, we know that
there exists $v \in V$ such that $\bar L_{i,m}v = \sigma_{i,m}v, \bar W_{i,m}v = \tau_{i,m}v$ for all $i,m\in\Z_+$, where $\sigma_{i,m},\tau_{i,m}\in\C$.

When  $\bar G_{0,0}v =0$, by the relation of $[\bar W_{i,m},\bar G_{0,0}]=\bar G_{i,m}$, we obtain $\bar G_{i,m}v=0$  for $i,m\in\Z_+$. Then $\mathrm{dim}(V)=\mathrm{dim}(V_{\bar0})=1$.

If
$\bar G_{0,0}v \neq0$, we present the following two cases.
First, consider
$\Phi_0=\emptyset$.
We have a decomposition of $\mathfrak{p}(t,N)$:
$$\mathfrak{p}(t,N)=\mathrm{span}_{\C} \{\bar L_{0,0},\bar W_{0,0},\bar G_{0,0}\}+\widetilde{\mathfrak{p}}(t,N),$$ where
$\widetilde{\mathfrak{p}}(t,N)=\mathfrak{p}(t,N)\setminus\mathrm{span}_{\C} \{\bar L_{0,0},\bar W_{0,0},\bar G_{0,0}\}$ and
$\widetilde{\mathfrak{p}}(t,N)$ is a nilpotent ideal of $\mathfrak{p}(t,N)$. Considering the action of
$\bar L_{0,0}$ on  $\widetilde{\mathfrak{p}}(t,N)$, which implies   that  $\widetilde{\mathfrak{p}}(t,N)$ is a completely reducible
$\C \bar L_{0,0}$-module with no trivial summand. By  Lemma $1$ of \cite{CK} and $\bar G_{0,0}^2v=0$,
we immediately obtain that $V_{\bar 1}=\C \bar G_{0,0}v$ and $\mathrm{dim}(V)=1+1$.

Next,   consider $\Phi_0\neq\emptyset$. Assume that $t,N\geq1$.
We note that if the  $\chi(t, N)$ (respectively,  ideals $\psi(t, N)$, $\phi(t, N)$) of  $\mathfrak{p}(t,N)$ acts trivially on $V$, then $V$ can be viewed as an irreducible module over $\mathfrak{p}(t-1,N)$ (respectively, $\mathfrak{p}(t,N-1)$, $\mathfrak{p}(t-1,N-1)$). Using simultaneous induction on $t$, $N$ and Lemma  \ref{cl2222}, we obtain that the odd vector can be written as $V_{\bar1}=\C \bar G_{0,0}v$, and $\mathrm{dim}(V)=1+1$.
\end{proof}

\section{Classification of finite irreducible modules}
This section will be  devoted to giving a complete  classification  of all finite non-trivial irreducible conformal modules over  $\mathcal{S}(p)$.
\subsection{Equivalence of modules}
We first recall a useful result  appeared in \cite{CHS}, which is related  to classification of  finite non-trivial irreducible conformal modules over
 $\mathcal{S}(p)_{\bar0}$.
\begin{lemm}\label{5.1}
Assume that  $V$ is a finite non-trivial irreducible conformal module over  $\mathcal{S}(p)_{\bar0}$. Then $V$ is isomorphic to
$V_{a,b,c,d}=\C[\partial]v$ with
  \begin{eqnarray*}
\\&&\begin{cases}
L_0\,{}_\lambda\, v=p(\partial+a\lambda+b)v,\\[4pt]
L_1\,{}_\lambda\,v=\delta_{p+1,0}cv,\\[4pt]
W_0\,{}_\lambda\, v=dv, \\[4pt]
W_i\,{}_\lambda\,v=0,\ i\geq1, \\[4pt]
L_j\,{}_\lambda\,v=0,\ j\geq2
\end{cases}\\
\end{eqnarray*} for  $a,b,c,d\in\C$.
\end{lemm}
The equivalence between finite conformal modules over $\mathcal{S}(p)$
 and those over its quotient algebra $\mathcal{S}(p)_{[n]}$ for  some $n\in\Z_+$ are given as follows.
 \begin{theo}\label{5.2}
Let  $V$ be a finite non-trivial conformal module over $\mathcal{S}(p)$.
Then the $\lambda$-actions  of $L_i$, $W_i$ and $G_i$ on $V$ are trivial for $i\gg0$.
\end{theo}
\begin{proof}
Regard $V$ as a finite  conformal module over  $\mathcal{S}(p)_{\bar 0}$.
It follows from  Theorem 4.2 of \cite{CHS} that
$L_i\,{}_\lambda\, v=W_i\,{}_\lambda\, v=0$  for all $i\gg0$ and any $v\in V$. Take $i$ such that $i>|p|$.
Fix $i\gg0$.  Since
 $$W_i\,{}_\lambda\,(G_0{}\,_\mu\, v)-G_0\,{}_\mu\,(W_i\,{}_\lambda\, v)=G_{i} \,{}_{\lambda+\mu}\, v,$$
one can check that  $G_{i} \,{}_{\lambda}\, v=0$
 for any $v\in V$, proving the theorem.
 \end{proof}

\begin{rema}
In fact, a finite conformal module over $\mathcal{S}(p)$  is isomorphic to  a   finite conformal module over
$\mathcal{S}(p)_{[n]}$  for some  large enough  $n\in\Z$, where $\mathcal{S}(p)_{[n]}$ is defined as \eqref{b3.2}.
\end{rema}

\subsection{Rank  $1+1$  modules}
In the following,  a characterization of    non-trivial  free   conformal modules of rank $1+1$  over
$\mathcal{S}(p)$ are presented.
 According to Lemma \ref{5.1}, we  can define  the following three classes of conformal modules
 $V_{a,b,c,d,a^\prime,b^\prime,c^\prime,d^\prime}$,
 $V_{a,b,c,d,\sigma}$ and   $V_{a,b,\sigma}$.
\begin{itemize}\parskip-7pt
 \item[{\rm (1)}] $V_{a,b,c,d,a^\prime,b^\prime,c^\prime,d^\prime}=\C[\partial]v_{\bar 0}\oplus\C[\partial]v_{\bar 1}$ with
 \begin{eqnarray*}
\\&&\begin{cases}
L_0\,{}_\lambda\, v_{\bar0}=p(\partial+a\lambda+b)v_{\bar0},\\[4pt]
 L_1\,{}_\lambda\,v_{\bar0}=c v_{\bar0},\\[4pt]
 W_0\,{}_\lambda\, v_{\bar0}=d v_{\bar0}, \\[4pt]
G_i\,{}_\lambda\,v_{\bar0}=0, \ i\geq0,\\[4pt]
W_j\,{}_\lambda\,v_{\bar0}=0,\ j\geq1, \\[4pt]
L_k\,{}_\lambda\,v_{\bar0}=0,\ k\geq2,
\end{cases}\   \mathrm{and} \
 \begin{cases}
L_0\,{}_\lambda\, v_{\bar1}=p(\partial+a^{\prime}\lambda+b^{\prime})v_{\bar1},\\[4pt]
 L_1\,{}_\lambda\,v_{\bar1}=c^{\prime} v_{\bar1},\\[4pt]
 W_0\,{}_\lambda\, v_{\bar1}=d^{\prime} v_{\bar1}, \\[4pt]
G_i\,{}_\lambda\,v_{\bar1}=0,\ i\geq0,\\[4pt]
W_j\,{}_\lambda\,v_{\bar1}=0,\ j\geq1,\\[4pt]
L_k\,{}_\lambda\,v_{\bar1}=0,\ k\geq2,
\end{cases}\\
\end{eqnarray*}
where $a,b,c,d,a^{\prime},b^{\prime},c^{\prime},d^{\prime}\in\C$;
 \item[{\rm (2)}] $V_{a,b,c,d,\sigma}=\C[\partial]v_{\bar 0}\oplus\C[\partial]v_{\bar 1}$ with
 \begin{eqnarray*}
\\&&\begin{cases}
L_0\,{}_\lambda\, v_{\bar0}=p(\partial+a\lambda+b)v_{\bar0},\\[4pt]
 L_1\,{}_\lambda\,v_{\bar0}=c v_{\bar0},\\[4pt]
 W_0\,{}_\lambda\, v_{\bar0}=d v_{\bar0}, \\[4pt]
  G_0\,{}_\lambda\, v_{\bar0}=\sigma v_{\bar1}, \\[4pt]
G_i\,{}_\lambda\,v_{\bar0}=0, \ i\geq1,\\[4pt]
W_j\,{}_\lambda\,v_{\bar0}=0,\ j\geq1, \\[4pt]
L_k\,{}_\lambda\,v_{\bar0}=0,\ k\geq2,
\end{cases}\   \mathrm{and} \
 \begin{cases}
L_0\,{}_\lambda\, v_{\bar1}=p(\partial+(a+1)\lambda+b)v_{\bar1},\\[4pt]
 L_1\,{}_\lambda\,v_{\bar1}=c v_{\bar1},\\[4pt]
 W_0\,{}_\lambda\, v_{\bar1}=(d+1) v_{\bar1}, \\[4pt]
G_i\,{}_\lambda\,v_{\bar1}=0,\ i\geq0,\\[4pt]
W_j\,{}_\lambda\,v_{\bar1}=0,\ j\geq1,\\[4pt]
L_k\,{}_\lambda\,v_{\bar1}=0,\ k\geq2,
\end{cases}\\
\end{eqnarray*}
where $a,b,c,d\in\C,\sigma\in\C^*$;
 \item[{\rm (3)}] $V_{a,b,\sigma}=\C[\partial]v_{\bar 0}\oplus\C[\partial]v_{\bar 1}$ with
  \begin{eqnarray*}
\\&&\begin{cases}
L_0\,{}_\lambda\, v_{\bar0}=p(\partial+a\lambda+b)v_{\bar0},\\[4pt]
 W_0\,{}_\lambda\, v_{\bar0}=(a-1) v_{\bar0}, \\[4pt]
  G_0\,{}_\lambda\, v_{\bar0}=\sigma\big(\partial+a\lambda+b\big) v_{\bar1}, \\[4pt]
G_i\,{}_\lambda\,v_{\bar0}=0,\ i\geq1,  \\[4pt]
W_j\,{}_\lambda\,v_{\bar0}=0,\ j\geq1, \\[4pt]
L_k\,{}_\lambda\,v_{\bar0}=0,\ k\geq1,
\end{cases}\  \mathrm{and} \
 \begin{cases}
L_0\,{}_\lambda\, v_{\bar1}=p(\partial+a\lambda+b)v_{\bar1},\\[4pt]
 W_0\,{}_\lambda\, v_{\bar1}=a v_{\bar1}, \\[4pt]
G_i\,{}_\lambda\,v_{\bar1}=0, \ i\geq0,\\[4pt]
W_j\,{}_\lambda\,v_{\bar1}=0, \ j\geq1, \\[4pt]
L_k\,{}_\lambda\,v_{\bar1}=0,\ k\geq1,
\end{cases}\\
\end{eqnarray*}
where $a,b\in\C,\sigma\in\C^*$.
\end{itemize}

\begin{theo}\label{77.1}
Let $V$ be a non-trivial free   conformal module of rank  $1+1$  over  $\mathcal{S}(p)$.
 \begin{itemize}\parskip-7pt
 \item[{\rm (1)}]
  If $p\neq-1$, then  $V\cong V_{a,b,0,d,a^\prime,b^\prime,0,d^\prime}$ or $V_{a,b,0,d,\sigma}$ or $V_{a,b,\sigma}$    or $\Pi({V}_{a-1,b,0,d-1,\sigma})$ or $\Pi({V}_{a,b,\sigma})$  for some $a,b,d,a^\prime,b^\prime,d^\prime\in\C$, $\sigma\in\C^*$$\mathrm{;}$
 \item[{\rm (2)}]
  If $p=-1$, then  $V\cong V_{a,b,c,d,a^\prime,b^\prime,c^\prime,d^\prime}$  or $V_{a,b,c,d,\sigma}$ or $V_{a,b,\sigma}$ or  $\Pi({V}_{a-1,b,c,d-1,\sigma})$ or $\Pi({V}_{a,b,\sigma})$  for some $a,b,c,d,a^\prime,b^\prime,c^\prime,d^\prime\in\C$, $\sigma\in\C^*$.
\end{itemize}
\end{theo}

\begin{proof}
In order to prove this theorem, we first let $V=\C[\partial]v_{\bar 0}\oplus\C[\partial]v_{\bar1}$.  Regarding $V$ as a conformal module over $\mathcal{S}(p)_{\bar 0}$ and using Lemma \ref{5.1},  we can suppose that
\begin{eqnarray*}
\\&&\begin{cases}
L_0\,{}_\lambda\, v_{\bar0}=p(\partial+a\lambda+b)v_{\bar0},\\[4pt]
 L_1\,{}_\lambda\,v_{\bar0}=\delta_{p+1,0}c v_{\bar0},\\[4pt]
 W_0\,{}_\lambda\, v_{\bar0}=d v_{\bar0}, \\[4pt]
W_i\,{}_\lambda\,v_{\bar0}=0,\ i\geq1, \\[4pt]
L_j\,{}_\lambda\,v_{\bar0}=0,\
j\geq2,
\end{cases}\ \mathrm{and}\
\begin{cases}
L_0\,{}_\lambda\, v_{\bar1}=p(\partial+a^\prime\lambda+b^\prime)v_{\bar1},\\[4pt]
 L_1\,{}_\lambda\,v_{\bar1}=\delta_{p+1,0}c^\prime v_{\bar1},\\[4pt]
 W_0\,{}_\lambda\, v_{\bar1}=d^\prime v_{\bar1}, \\[4pt]
W_i\,{}_\lambda\,v_{\bar1}=0,\ i\geq1,\\[4pt]
L_j\,{}_\lambda\,v_{\bar1}=0,\ j\geq2,
\end{cases}\\
\end{eqnarray*}
where $a,b,c,d,a^\prime,b^\prime,c^\prime,d^\prime\in\C$.
Owing to  Theorem \ref{5.2}, we see that   $L_i\,{}_\lambda\, v_s=W_i\,{}_\lambda\, v_s=G_i\,{}_\lambda\, v_s=0$   for $s\in\Z_2,i\gg0$.
Let $t_0\in\Z_+$ be the largest integer such that the action of $\mathcal{S}(p)_{t_0}$  (see \eqref{b3547877wq72}) on $V$ is non-trivial.
For $0\leq t\leq t_0,t\in\Z_+$, we can write
\begin{eqnarray*}
&&L_t\,{}_\lambda\, v_{\bar0}=f_t(\partial,\lambda)v_{\bar0}, \ L_t\,{}_\lambda\, v_{\bar1}=\widehat{f}_t(\partial,\lambda)v_{\bar1},
\\&& W_t\,{}_\lambda\, v_{\bar0}=g_t(\partial,\lambda)v_{\bar0}, \ W_t\,{}_\lambda\, v_{\bar1}=\widehat{g}_t(\partial,\lambda)v_{\bar1},
\\&& G_t\,{}_\lambda\, v_{\bar0}=h_t(\partial,\lambda)v_{\bar1}, \ G_t\,{}_\lambda\, v_{\bar1}=\widehat{h}_t(\partial,\lambda)v_{\bar0},
\end{eqnarray*}
where $f_t(\partial,\lambda),\widehat{f}_t(\partial,\lambda),g_t(\partial,\lambda),\widehat{g}_t(\partial,\lambda),
h_t(\partial,\lambda),\widehat{h}_t(\partial,\lambda)\in\C[\partial, \lambda]$.
Note that
\begin{eqnarray*}
&&f_0(\partial,\lambda)=p(\partial+a\lambda+b),\ f_1(\partial,\lambda)=\delta_{p+1,0}c,f_i(\partial,\lambda)=0,\
\\&& \widehat{f}_0(\partial,\lambda)=p(\partial+a^\prime\lambda+b^\prime),\ \widehat{f}_1(\partial,\lambda)=\delta_{p+1,0}c^\prime,\widehat{f}_i(\partial,\lambda)=0,\
\\&&g_0(\partial,\lambda)=d,\ \widehat{g}_0(\partial,\lambda)=d^\prime,\
g_j(\partial,\lambda)=\widehat{g}_j(\partial,\lambda)=0
  \end{eqnarray*} for $i\geq2,j\geq1$.
For $1\leq t\leq t_0,s\in\Z_2$, we have
\begin{eqnarray*}
   W_t\,{}_\lambda\, (G_0{}\,_\mu\, v_{s})-G_0\,{}_\mu\,(W_t\,{}_\lambda\, v_{s})=G_t\,{}_{\lambda+\mu}\, v_{s},
\end{eqnarray*}
which gives
$h_t(\partial,\lambda)=\widehat{h}_t(\partial,\lambda)=0$ for $t\geq1$.
Then for $s\in\Z_2$, by
$L_1\,{}_\lambda\, (G_0{}\,_\mu\, v_{s})=G_0\,{}_\mu\,(L_1\,{}_\lambda\, v_{s}),$
one has
\begin{eqnarray}\label{d41}
\big(h_0(\partial+\lambda,\mu)c^\prime-h_0(\partial,\mu)c\big)\delta_{p+1,0}=0 \quad \mathrm{and}  \quad \big(\widehat{h}_0(\partial+\lambda,\mu)c-\widehat{h}_0(\partial,\mu)c^\prime\big)\delta_{p+1,0}=0.
\end{eqnarray}
  For $s\in\Z_2$, we have
$G_0\,{}_\lambda\, (G_0{}\,_\mu\, v_{s})+G_0\,{}_\mu\,(G_0\,{}_\lambda\, v_{s})=0,$
which implies
\begin{eqnarray}
&& \label{6.5}    h_{0}(\partial+\lambda,\mu)\widehat{h}_{0}(\partial,\lambda)
+h_{0}(\partial+\mu,\lambda)\widehat{h}_{0}(\partial,\mu)=0,
 \\&&
 \label{6.6}  \widehat{h}_{0}(\partial+\lambda,\mu)h_{0}(\partial,\lambda)
 +\widehat{h}_{0}(\partial+\mu,\lambda)h_{0}(\partial,\mu)=0.
\end{eqnarray}
Setting $\lambda=\mu=0$ in \eqref{6.5} or \eqref{6.6}, one can get
$h_{0}(\partial,0)\widehat{h}_{0}(\partial,0)=0$.
Let us consider the following three cases.
\begin{case}
$h_{0}(\partial,0)=\widehat{h}_{0}(\partial,0)=0$.
\end{case}
For $s\in\Z_2$, we get
$L_0\,{}_\lambda\, (G_0{}\,_\mu\, v_{s})-G_0\,{}_\mu\,(L_0\,{}_\lambda\, v_{s})=[L_0\,{}_\lambda\, G_0]\,{}_{\lambda+\mu}\, v_{s},$
which shows
\begin{eqnarray}
&&\nonumber h_{0}(\partial+\lambda,\mu)(\partial+a^\prime\lambda+b^\prime)-(\partial+\mu+a\lambda+b)h_{0}(\partial,\mu)
\\&=&\label{143} \big(\lambda-\mu\big)h_0(\partial,\lambda+\mu),
\\&&\nonumber \widehat{h}_{0}(\partial+\lambda,\mu)(\partial+a\lambda+b)
-(\partial+\mu+a^\prime\lambda+b^\prime)\widehat{h}_{0}(\partial,\mu)
\\&=&\label{144}   \big(\lambda-\mu\big)\widehat{h}_0(\partial,\lambda+\mu).
\end{eqnarray}
Setting $\mu=0$ in \eqref{143} and \eqref{144}, we check that   $h_{0}(\partial,\lambda)=\widehat{h}_{0}(\partial,\lambda)=0$.
\begin{case}\label{case4}
$h_{0}(\partial,0)\neq0$  \ $\mathrm{and}$ \ $\widehat{h}_{0}(\partial,0)=0$.
\end{case}
Taking  $\mu=0$ in  \eqref{144} gives    $\widehat{h}_{0}(\partial,\lambda)=0$.   Based on
$W_0\,{}_\lambda\, (G_0{}\,_\mu\, v_{\bar0})-G_0\,{}_\mu\,(W_0\,{}_\lambda\, v_{\bar0})=G_0\,{}_{\lambda+\mu}\, v_{\bar0},$
we obtain
\begin{eqnarray}\label{4.55}
d^\prime h_0(\partial+\lambda,\mu)-dh_0(\partial,\mu)=h_0(\partial,\lambda+\mu).
\end{eqnarray}
   Setting $\lambda=\mu=0$ respectively in  \eqref{143} and \eqref{4.55}, we check that
$$b^\prime=b \quad \mathrm{and} \quad d^\prime=d+1.$$ Then we  let   $\mu=0$ in \eqref{143} and \eqref{4.55}, which gives
\begin{eqnarray}
&& \label{479}
 (\partial+a^\prime\lambda+b^\prime)h_{0}(\partial+\lambda,0)- (\partial +a\lambda+b)h_{0}(\partial,0)
 = \lambda h_0(\partial,\lambda),
\\&&  \label{478} d^\prime h_0(\partial+\lambda,0)-dh_0(\partial,0)=h_0(\partial,\lambda).
 \end{eqnarray}
 Inserting \eqref{478} into \eqref{479}, we get
  \begin{eqnarray}\label{578}
  (\partial+(a^\prime-d^\prime)\lambda+b)h_{0}(\partial+\lambda,0)=(\partial +(a-d)\lambda+b)h_{0}(\partial,0).
\end{eqnarray}
Consider
$a^\prime\neq d^\prime,a\neq d$  or $a^\prime= d^\prime,a= d$ in \eqref{578}. It is straightforward to verify that $h_{0}(\partial,0)=\sigma\in\C^*.$
Using this in \eqref{479} and \eqref{478}, we have
$$h_0(\partial,\lambda)=\sigma \quad \mathrm{and} \quad a^\prime=a+1.$$
If $p=-1$ in \eqref{d41}, one has $c=c^\prime$.

Clearly, we need not   discuss the case for  $a^\prime\neq d^\prime$ and $a=d$.
The final case  is $a^\prime=d^\prime$ and $a\neq d$.   Considering the highest degree of $\lambda$ in \eqref{578}, one can obtain $\mathrm{deg}_{\lambda}\big(h_{0}(\partial+\lambda,0)\big)=1$.
The equation of \eqref{479} can be written as
 \begin{eqnarray}\label{146}
h_{0}(\partial,\lambda)+ah_{0}(\partial,0)-a^\prime h_{0}(\partial+\lambda,0)
= \big(\partial+b\big)\frac{h_0(\partial+\lambda,0)-h_0(\partial,0)}{\lambda}.
\end{eqnarray}
 Taking $\lambda\rightarrow0$ in \eqref{146}, we get
$(a-a^\prime+1)h_{0}(\partial,0)
= \big(\partial+b\big)\frac{d}{d\partial}\big(h_{0}(\partial,0)\big),$
which implies
   \begin{eqnarray}\label{148}
 a^\prime=a\quad \mathrm{and}  \quad h_{0}(\partial,0)=\sigma(\partial+b)
     \end{eqnarray} for $\sigma\in\C^*$. Now  putting \eqref{148} in \eqref{479}  or \eqref{478} gives
 $$h_{0}(\partial,\lambda)=\sigma (\partial+a\lambda+b).$$
If $p=-1$ in \eqref{d41}, we have $c=c^\prime=0$.
\begin{case}
$h_{0}(\partial,0)=0$ \ $\mathrm{and}$ \ $\widehat{h}_{0}(\partial,0)\neq0$.
\end{case}
By the similar arguments  in Case \ref{case4},
we show that
$$a^\prime=a-1,b^\prime=b,d^\prime=d-1,\widehat{h}_{0}(\partial,\lambda)=\sigma \ \mathrm{and\ if}\ p=-1, c^\prime=c,$$
or
$$a^\prime=a=d,b^\prime=b,d^\prime=a-1,
\widehat{h}_{0}(\partial,\lambda)=\sigma(\partial+a\lambda+b)\ \mathrm{and\ if}\ p=-1, c^\prime=c=0,$$ where $\sigma\in\C^*$.
 This completes the proof.
\end{proof}

Now we  determine the irreducibilities of conformal  modules $V$ over $\mathcal{S}(p)$ defined  in Theorem  \ref{77.1}.
\begin{prop}\label{pro77.111}
Let $V$ be a   conformal module    over  $\mathcal{S}(p)$ defined  in Theorem  \ref{77.1}.
\begin{itemize}\parskip-7pt
 \item[{\rm (1)}]
If $V\cong V_{a,b,c,d,a^\prime,b^\prime,c^\prime,d^\prime}$,  then $V$ is reducible, which has some submodules as   $\C[\partial] v_{\bar0}$ and $\C[\partial]v_{\bar 1}$.
 \item[{\rm (2)}]
If $V\cong V_{a,b,c,d,\sigma}$,  then $V$ is irreducible  if and only if $(a,c,d)\neq(0,0,0)$.  The module $V_{0,b,0,0,\sigma}$ contains a unique non-trivial submodule $\C[\partial](\partial+b)v_{\bar0}\oplus\C[\partial]v_{\bar 1}\cong {V}_{1,b,\sigma}$.
 \item[{\rm (3)}]
If  $V\cong  {V}_{a,b,\sigma}$,  then $V$ is irreducible  if and only if $a\neq0$. The module  $V_{0,b,\sigma}$ contains a unique non-trivial submodule $\C[\partial]v_{\bar0}\oplus\C[\partial](\partial+b)v_{\bar 1}\cong  V_{0,b,0,-1,\sigma}$.
  \item[{\rm (4)}]
If $V\cong \Pi({V}_{a-1,b,c,d-1,\sigma})$,  then $V$ is irreducible  if and only if $(a,c,d)\neq(1,0,1)$.  The module $\Pi(V_{0,b,0,0,\sigma})$ contains a unique non-trivial submodule $\C[\partial]v_{\bar0}\oplus\C[\partial](\partial+b)v_{\bar 1}\cong  \Pi(V_{1,b,\sigma})$.
 \item[{\rm (5)}]
If  $V\cong \Pi({V}_{a,b,\sigma})$,  then $V$ is irreducible  if and only if $a\neq0$. The module  $\Pi(V_{0,b,\sigma})$ contains a unique non-trivial submodule $\C[\partial](\partial+b)v_{\bar0}\oplus\C[\partial]v_{\bar 1}\cong \Pi(V_{0,b,0,-1,\sigma})$.
\end{itemize}
\end{prop}

\subsection{Classification theorems}
 The following lemma can be found in  \cite{CK,K3}.
\begin{lemm}\label{lemm2.5}
 Let $\W$ be a Lie superalgebra with a descending sequence of subspaces  $\W\supset\W_0\supset\W_1\supset\cdots$ and an element $\partial$ satisfying $[\partial, \W_n] =\W_{n-1}$ for $n\geq1$. Let $V$ be an $\W$-module and let
$$V_n=\{v\in V | \W_n v =0\},\  n\in\Z_+.$$
Suppose that $V_n\neq0$ for $n\gg0$  and let $N$ denote the minimal such $n$. Suppose that $N\geq1$. Then $V=\C[\partial]\otimes_{\C} V_N$. Particularly, $V_N$ is finite-dimensional
 if $V$ is a finitely generated $\C[\partial]$-module.
\end{lemm}
The method  in  Lemma 6.3 of \cite{X} can be expanded to the following results  with a slightly different discussion.
\begin{lemm}\label{lemm789}
Any finite non-trivial irreducible   $\mathcal{S}(p)$-module $V$ must be free of rank $1$ or $1+1$.
\end{lemm}
\begin{proof}
It follows from any torsion module of $\C[\partial]$ is trivial as a module of Lie conformal superalgebra  that any finite non-trivial irreducible   $\mathcal{S}(p)$-module $V$ must be free as a $\C[\partial]$-module.
According to   Theorem \ref{5.2}, one can see that the $\lambda$-actions  of $L_i,$ $W_i$ and $G_i$ on $V$ are trivial for all $i\gg0$.
  Assume that $t\in\Z_+$ is the largest integer such that the $\lambda$-action  of $\mathcal{S}(p)_t$  on $V$ is non-trivial.
Thus $V$ can be regarded as a    finite non-trivial irreducible  conformal module over $\mathcal{S}(p)_{[t]}$.
Denote
$${\mathfrak{gsb}}=\big\{\bar L_{i,m},\bar W_{j,n},\bar G_{k,l},\partial\mid 0\leq i,j,k\leq t, n\in\Z_+,m,l\in\Z_+\cup\{-1\}\big\}.$$
Based on  Proposition \ref{pro2.4},
  the conformal  $\mathcal{S}(p)_{[t]}$-module  $V$ can be viewed as a module over the associated extended annihilation algebra
   $\mathfrak{gsb}=\mathcal{A}(\mathcal{S}(p)_{[t]})^e$, which   satisfies
\begin{eqnarray}\label{7.1}
\bar L_{i,m}v=\bar W_{j,n}v=\bar G_{k,l}v=0
\end{eqnarray}
for $0\leq i,j,k\leq t$, $m,n,l\gg0$, $v\in V$.
For later use, we write
$$\mathfrak{gsb}_z=\{\bar L_{i,m},\bar W_{j,n},\bar G_{k,l}\in\mathfrak{gsb}\mid 0\leq i,j,k\leq t,m,l\geq z-1,n\geq z\},\ \forall z\in\Z_+.$$
Then   $\mathfrak{gsb}_0=\mathcal{A}(\mathcal{S}(p)_{[t]})$   and $\mathfrak{gsb}\supset\mathfrak{gsb}_0\supset\mathfrak{gsb}_1\supset\cdots$.
It follows from the definition of extended annihilation algebra that the element  $\partial\in\mathfrak{gsb}$ satisfies $[\partial, \mathfrak{gsb}_z] =\mathfrak{gsb}_{z-1}$ for $z \geq1$. Denote
$$V_z=\{v\in V\mid \mathfrak{gsb}_zv=0\},\ \forall z\in\Z_+.$$
We observe that  $V_z\neq\emptyset$ for $z\gg0$ by \eqref{7.1}.  Assume that  $N\in \Z_+$ is the smallest integer such that $V_N\neq\emptyset$.

Firstly,   consider  $N=0$. Let $0\neq v\in V_0$. Then $\mathcal{U}(\mathfrak{gsb})v=\C[\partial]\mathcal{U}(\mathfrak{gsb}_0)v=\C[\partial]v$. Therefore,
from the irreducibility of $V$, we have  $V=\C[\partial]v$.
It is clear that $\mathfrak{gsb}_0$ acts trivially on $V$.
According to  Proposition \ref{pro2.4}, we know  that $V$ is a   trivial conformal $\mathcal{S}(p)$-module, which leads to a contradiction.

Secondly, consider  $N\geq1$.  Take $0\neq v\in V_N$.
From $$[\partial-\frac{1}{p}\bar L_{0,-1},L_{i,m}]=[\partial-\frac{1}{p}\bar L_{0,-1},W_{i,m}]=[\partial-\frac{1}{p}\bar L_{0,-1},G_{i,m}]=0$$
for $i,m\in\Z_+$, one can show that
$\partial-\frac{1}{p}\bar L_{0,-1}$ is an even central
element of $\mathfrak{gsb}$ by the definition of extended annihilation algebra. So there exists some $\varrho\in\C$ such that   $\bar L_{0,-1}v=p(\partial+\varrho)v$  by  Schur's Lemma of Lie superalgebra.
Moreover, according to  the following relations
 $$\bar L_{i,-1}v=\frac{1}{p}[\bar L_{i,0},\bar L_{0,-1}]v\  \mathrm{and} \  \bar G_{i,-1}v=\frac{1}{p}[\bar G_{i,0},\bar L_{0,-1}]v,$$
 we obtain that the action of $\mathfrak{gsb}_0$ on  $v$ is determined by $\mathfrak{gsb}_1$ and
$\partial$. Obviously, $V_N$ is $\mathfrak{gsb}_1$-invariant. By the irreducibility of $V$ and Lemma \ref{lemm2.5},
 we obtain that $V=\C[\partial]\otimes_{\C}V_N$  and $V_N$ is a non-trivial irreducible finite-dimensional $\mathfrak{gsb}_1$-module.

If $N=1$, from the definition of $V_1$,
     we  know that $V_1$ is a trivial $\mathfrak{gsb}_1$-module,   which gives a   contradiction.

If $N\geq 2$,   the module   $V_N$   can be viewed as a $\mathfrak{gsb}_1/ \mathfrak{gsb}_N$-module.
Note that $\mathfrak{gsb}_1/ \mathfrak{gsb}_N\cong\mathfrak{p}({t,N-2})$. Based on Lemma \ref{lemm555},
 one can see that   $V_N$ is $1$-dimensional or $1+1$-dimensional. Then  $V$ is free of rank $1$ or $1+1$  as a conformal module over $\mathcal{S}(p)$  by Proposition \ref{pro2.4}.
The lemma holds.
\end{proof}

Note that the $\mathcal{S}(p)_{\bar 0}$-module $V_{a,b,c,d}$ defined in Lemma \ref{5.1} can be regarded as
 an $\mathcal{S}(p)$-module by the trivial action  of $\mathcal{S}(p)_{\bar 1}$. Combining  Theorem \ref{77.1}, Proposition \ref{pro77.111} and Lemma \ref{lemm789},  we have the main result  of this paper as follows, which shows that the irreducible modules $V$  defined in Theorem  \ref{77.1} exhaust
  all non-trivial finite irreducible conformal modules over $\mathcal{S}(p)$.

\begin{theo}\label{7.111}
Let $V$ be a non-trivial finite irreducible conformal module over  $\mathcal{S}(p)$.
 \begin{itemize}\parskip-7pt
 \item[{\rm (1)}]
  If $p\neq-1$, then
  \begin{eqnarray*}
V\cong
\begin{cases}
V_{a,b,0,d} &\ \mbox{for}\  (a,d)\neq(0,0),b\in\C,\\[4pt]
V_{a,b,0,d,\sigma} &\ \mbox{for}\  (a,d)\neq(0,0),b\in\C,\sigma\in\C^*,\\[4pt]
V_{a,b,\sigma}&\  \mbox{for} \ a,\sigma\in\C^*,b\in\C,\\[4pt]
\Pi({V}_{a-1,b,0,d-1,\sigma})&\  \mbox{for} \ (a,d)\neq(1,1),b\in\C,\sigma\in\C^*,\\[4pt]
\Pi({V}_{a,b,\sigma})&\  \mbox{for} \ a,\sigma\in\C^*,b\in\C;
\end{cases}\\
\end{eqnarray*}
 \item[{\rm (2)}]
 If $p=-1$, then
    \begin{eqnarray*}
V\cong
\begin{cases}
V_{a,b,c,d} &\ \mbox{for}\  (a,c,d)\neq(0,0,0),b\in\C,\\[4pt]
V_{a,b,c,d,\sigma} &\ \mbox{for}\  (a,c,d)\neq(0,0,0),b\in\C,\sigma\in\C^*,\\[4pt]
V_{a,b,\sigma}&\  \mbox{for} \ a,\sigma\in\C^*,b\in\C,\\[4pt]
\Pi({V}_{a-1,b,c,d-1,\sigma})&\  \mbox{for} \ (a,c,d)\neq(1,0,1),b\in\C,\sigma\in\C^*,\\[4pt]
\Pi({V}_{a,b,\sigma})&\  \mbox{for} \ a,\sigma\in\C^*,b\in\C.
\end{cases}\\
\end{eqnarray*}
\end{itemize}
\end{theo}

\section{Generalized version  of $\mathcal{S}(p)$}
The aim of this section is to   construct a class of new  Lie conformal superalgebras
as  an  extended case  of $\mathcal{S}(p)$.

Now,   we  continue to apply the     $\mathfrak{B}(p,0,p)$-module $V(\alpha_1,\beta_1,\gamma_1,p)$   in \eqref{v111}  and define a   super module $V=V_{\bar 0}\oplus V_{\bar 1}$, where $V_{\bar0}=\oplus_{i\in\Z}\C[\partial]v_{\bar0}^i$ and $V_{\bar1}=\oplus_{i\in\Z}\C[\partial]v_{\bar1}^i$.
More precisely, for  $\alpha_1,\beta_1,\gamma_1,\widehat{\alpha}_1,\widehat{\beta}_1,\widehat{\gamma_1}\in\C$, $p \in\C^*$,   the   $\C[\partial]$-module $$V(\alpha_1,\beta_1,\gamma_1,\widehat{\alpha}_1,\widehat{\beta}_1,\widehat{\gamma_1},p)
=\oplus_{i\in\Z}\C[\partial]v_{\bar0}^i\bigoplus\oplus_{i\in\Z}\C[\partial]v_{\bar1}^i$$  is a $\Z$-graded free intermediate
series module of rank two over $\mathfrak{B}(p,0,p)$ with  $\lambda$-actions as follows:
\begin{eqnarray*}
 &&L_i\, {}_\lambda \, v_{\bar0}^j=\Big((i+p)(\partial+\beta_1)+(i+j+\alpha_1)\lambda\Big)v_{\bar0}^{i+j},\
 W_i\, {}_\lambda \, v_{\bar0}^j=\gamma_1 v_{\bar0}^{i+j},
 \\&&L_i\, {}_\lambda \, v_{\bar1}^j=\Big((i+p)(\partial+\widehat{\beta}_1)+(i+j+\widehat{\alpha}_1)\lambda\Big)v_{\bar1}^{i+j},\
 W_i\, {}_\lambda \, v_{\bar1}^j=\widehat{\gamma_1} v_{\bar1}^{i+j}.
\end{eqnarray*}
We consider a $\Z_2$-graded $\C[\partial]$-module
$$\mathcal{GS}(\alpha_1,\beta_1,\gamma_1,\widehat{\alpha}_1,\widehat{\beta}_1,
\widehat{\gamma_1},\{\phi_{i,j},\varphi_{i,j}\})=\mathcal{GS}_{\bar 0}\oplus\mathcal{GS}_{\bar 1}$$
with
$\mathcal{GS}_{\bar 0}=\oplus_{i\in\Z_+}\C[\partial]L_i\bigoplus\oplus_{i\in\Z_+}\C[\partial]W_i$,
$\mathcal{GS}_{\bar 1}=\oplus_{i\in\Z_+}\C[\partial]G_i\bigoplus\oplus_{i\in\Z_+}\C[\partial]H_i$ and satisfying
\begin{eqnarray*}
&&[L_i\, {}_\lambda \, L_j]=\big((i+p)\partial+(i+j+2p)\lambda \big) L_{i+j},
\\&& [L_i\, {}_\lambda \, W_j]=\Big((i+p)(\partial+\beta)+(i+j+\alpha)\lambda\Big) W_{i+j}
\\&&  [L_i\, {}_\lambda \, G_j]=\Big((i+p)(\partial+\beta_1)+(i+j+\alpha_1)\lambda\Big) G_{i+j},
\\&&  [W_i\, {}_\lambda \, G_j]=\gamma_1 G_{i+j},
\\&&  [L_i\, {}_\lambda \, H_j]=\Big((i+p)(\partial+\widehat{\beta}_1)+(i+j+\widehat{\alpha}_1)\lambda\Big) H_{i+j},
\\&&  [W_i\, {}_\lambda \, H_j]=\widehat{\gamma_1} H_{i+j},
\\&& [G_i\, {}_\lambda \, H_j]=\phi_{i,j}(\partial,\lambda)W_{i+j}+\varphi_{i,j}(\partial,\lambda)L_{i+j},
\\&& [W_i\, {}_\lambda \, W_j]= [G_i\, {}_\lambda \, G_j]= [H_i\, {}_\lambda \, H_j]=0,
\end{eqnarray*}
where  $\phi_{i,j}(\partial,\lambda),\varphi_{i,j}(\partial,\lambda)\in\C[\partial,\lambda]$  for $i,j\in\Z_+$.
\begin{lemm}\label{le51}
 Let $p\in\C^*,\alpha_1=2p,\widehat{\alpha}_1=p,\beta_1=\widehat{\beta}_1=0,\gamma_1=1$ and $\widehat{\gamma}_1=-1$.  Then $\Z_2$-graded $\C[\partial]$-module  $\mathcal{GS}(\alpha_1,\beta_1,\gamma_1,\widehat{\alpha}_1,\widehat{\beta}_1,
\widehat{\gamma}_1,\{\phi_{i,j},\varphi_{i,j}\})$ becomes a Lie conformal superalgebra if and only if
 $\phi_{i,j}(\partial,\lambda)=\Delta$ and $\varphi_{i,j}(\partial,\lambda)=\Delta\big((i+p)\partial+(i+j+p)\lambda\big)$ for $i,j\in\Z_+,\Delta\in\C$.
\end{lemm}
\begin{proof}
It follows from Definition \ref{D1} that the sufficiency is clear.

 We only need to prove the necessity.
   Assume that $\mathcal{GS}(\alpha_1,\beta_1,\gamma_1,\widehat{\alpha}_1,\widehat{\beta}_1,
\widehat{\gamma}_1,\{\phi_{i,j},\varphi_{i,j}\})$ is a Lie conformal superalgebra.
   For any $i, j\in \Z_+$, using the Jacobi identity for triple $(W_0, G_i, H_j)$, we have
   \begin{eqnarray}
   \label{1555.6} &&\varphi_{i,j}(\partial,\lambda+\mu)=\varphi_{i,j}(\partial,\mu),
\\&& \label{1555.66}(i+j+p)\lambda\varphi_{i,j}(\partial+\lambda,\mu)
  =\phi_{i,j}(\partial,\lambda+\mu)-\phi_{i,j}(\partial,\mu).
\end{eqnarray}
From  \eqref{1555.6}, one has
\begin{eqnarray}
   \label{1555.699} \varphi_{i,j}(\partial,\lambda)=\varphi_{i,j}(\partial,0)
\end{eqnarray}
for $i,j\in\Z_+.$
  For any $i, j\in \Z_+$, by the Jacobi identity for triple $(L_0, G_i, H_j)$, we check that
   \begin{eqnarray}
 \label{555.6} \nonumber &&\big(p\partial+(i+j+p)\lambda\big)\phi_{i,j}(\partial+\lambda,\mu)
   \\&=&\big((i+p)\lambda-p\mu\big)\phi_{i,j}(\partial,\lambda+\mu)
  +\big(p(\partial+\mu)+(j+p)\lambda\big)\phi_{i,j}(\partial,\mu),
  \\&& \nonumber  \label{555.67}  \big(p\partial+(i+j+2p)\lambda\big)\varphi_{i,j}(\partial+\lambda,\mu)
   \\&=&\big((i+p)\lambda-p\mu\big)\varphi_{i,j}(\partial,\lambda+\mu)
  +\big(p(\partial+\mu)+(j+p)\lambda\big)\varphi_{i,j}(\partial,\mu).
\end{eqnarray}
 Setting $\mu=0$ in \eqref{555.67} and  using     \eqref{1555.699},
 we conclude that
 $\varphi_{i,j}(\partial,\lambda)=\Delta\in\C$ for  any $i,j\in\Z_+$.
We take $\mu=0$ in \eqref{1555.66},
one can write
 \begin{eqnarray}\label{p566}
\phi_{i,j}(\partial,\lambda)=\phi_{i,j}(\partial,0)+ \Delta(i+j+p)\lambda
\end{eqnarray}
By choosing $\mu=0$ in \eqref{555.6}, it is easy to show that
 \begin{eqnarray} \label{76y}
 \nonumber && p\partial\frac{\big(\phi_{i,j}(\partial+\lambda,0)-\phi_{i,j}(\partial,0)\big)}{\lambda}
   \\&=&(i+p)\phi_{i,j}(\partial,\lambda)
  + (j+ p)\phi_{i,j}(\partial,0)-(i+j+p)\phi_{i,j}(\partial+\lambda,0).
\end{eqnarray}
Taking $\lambda\rightarrow0$, we have $\partial\frac{d}{d\partial}\phi_{i,j}(\partial,0)=\phi_{i,j}(\partial,0)$,
which has a nonzero solution $\phi_{i,j}(\partial,0)= \Delta_0(i+p)\partial$ for $\Delta_0\in\C$.
Then \eqref{p566} can be written as
 \begin{eqnarray}\label{p577}
\phi_{i,j}(\partial,\lambda)=\Delta_0(i+p)\partial+ \Delta(i+j+p)\lambda
\end{eqnarray}
Inserting \eqref{p577}
into \eqref{76y}, we immediately  obtain $\Delta=\Delta_0,$
which yields
$\phi_{i,j}(\partial,\lambda)=\Delta\big((i+p)\partial+(i+j+p)\lambda\big)$ for $i,j\in\Z_+,\Delta\in\C$. We complete the proof.

\end{proof}

\begin{rema}
Up to isomorphism, we may assume that $\Delta=2$
  in  Lemma \ref{le51} for  $\Delta\neq0$. Then we can define a class of Lie conformal superalgebras $\mathcal{GS}(p)=\mathcal{GS}_{\bar 0}\oplus\mathcal{GS}_{\bar 1}$
with $$\mathcal{GS}_{\bar 0}=\oplus_{i\in\Z_+}\C[\partial]L_i\bigoplus\oplus_{i\in\Z_+}\C[\partial]W_i,\ \mathcal{GS}_{\bar 1}=\oplus_{i\in\Z_+}\C[\partial]G_i\bigoplus\oplus_{i\in\Z_+}\C[\partial]H_i$$
and the following non-trivial $\lambda$-brackets
 \begin{eqnarray*}
&&[L_i\, {}_\lambda \, L_j]=\big((i+p)\partial+(i+j+2p)\lambda \big) L_{i+j},
\\&&[L_i\, {}_\lambda \, W_j]=\Big((i+p)\partial+(i+j+p)\lambda\Big) W_{i+j},
\\&&[L_i\, {}_\lambda \, G_j]=\Big((i+p)\partial+(i+j+2p)\lambda\Big) G_{i+j},
\\&&[L_i\, {}_\lambda \, H_j]=\Big((i+p)\partial+(i+j+p)\lambda\Big) H_{i+j},
\\&&  [W_i\, {}_\lambda \, G_j]=G_{i+j},\ [W_i\, {}_\lambda \, H_j]=-H_{i+j},
 \\&&[G_i\, {}_\lambda \, H_j]=2L_{i+j}+2\Big((i+p)\partial+(i+j+p)\lambda\Big)W_{i+j}.
\end{eqnarray*}
Denote the subalgebra
 $\C[\partial]L_0\oplus\C[\partial]W_0\oplus\C[\partial]G_0\oplus\C[\partial]H_0$ of $\mathcal{GS}(p)$ by
 $\mathcal{SN}$.
Define the following $\C[\partial]$-module homomorphism from $\mathcal{SN}$ to the Lie conformal algebra of $N=2$ superconformal algebra $\mathrm{(}$see \cite{CL}$\mathrm{)}$:
$$\frac{1}{p}L_0+\frac{1}{2}\partial W_0\rightarrow L,\ W_0\rightarrow J,\ G_0\rightarrow G_+,\ \frac{1}{p}H_0\rightarrow G_-.$$
It is easy to check that the subalgebra
 $\mathcal{SN}$ is isomorphic to the Lie conformal algebra of $N=2$ superconformal algebra.

\end{rema}
   A class of infinite-dimensional  Lie superalgebras related to Block type Lie algebra are presented in the rest of this section, which have a subalgebra called topological $N=2$ superconformal algebra.
\begin{lemm}
 The annihilation  superalgebra  of $\mathcal{GS}(p)$ is given by
$$\mathcal{A}(\mathcal{GS}(p))=\Big\{L_{i,m},W_{j,n},G_{k,l},H_{p,q}\mid i,j,k,n\in \Z_+,m,l\in\Z_+\cup\{-1\}\Big\}$$
with    non-vanishing    relations:
\begin{equation*}
\aligned
&[L_{i,m},L_{j,n}]=\big((m+1)(j+p)-(n+1)(i+p)\big)L_{i+j,m+n},
\\&[L_{i,m},W_{j,n}]=\big((m+1)j-n(i+p)\big)W_{i+j,m+n},
\\&[L_{i,m},G_{j,n}]=\big((m+1)(j+p)-(n+1)(i+p)\big)G_{i+j,m+n},
\\&[L_{i,m},H_{j,n}]=\big((m+1)j-n(i+p)\big)H_{i+j,m+n},
\\&  [G_{i,m},H_{j,n}]=2L_{i+j,m+n}+2\big((m+1)j-n(i+p)\big)W_{i+j,m+n},
\\&  [W_{i,m},G_{j,n}]=G_{i+j,m+n},
\ [W_{i,m},H_{j,n}]=-H_{i+j,m+n},
\endaligned
\end{equation*}where $p\in\C^*$.
\end{lemm}
\begin{proof}
Since the proof  is similar to Lemma  \ref{4.1}, we omit the details.
\end{proof}
\begin{rema}
We observe  that
  the  topological $N=2$ superconformal algebra  is isomorphic
to the Lie superalgebra   generated by
$\{L_{0,m},W_{0,n},G_{0,l},H_{0,q}\mid m,n,l,q\in\Z\}$ $\mathrm{(}$see \cite{DVV}$\mathrm{)}$.
\end{rema}

\section{Applications}
From  the definition  of  \eqref{b3.2},  we know that  $\mathcal{S}(p)_{[n]}$ has a  $\C[\partial]$-basis $\{\bar L_{i},\bar W_{i}, \bar G_i\mid 0\leq i\leq n\}$ with  the following
  $\lambda$-brackets:
\begin{eqnarray*}
&& [\bar L_i\, {}_\lambda \,\bar L_j]=\big((i+p)\partial+(i+j+2p)\lambda \big)\bar L_{i+j},
\\&&
[\bar L_i\, {}_\lambda \, \bar W_j]=\Big((i+p)\partial+(i+j+p)\lambda\Big) \bar W_{i+j},
\\&&[\bar L_i\, {}_\lambda \, \bar G_j]=\Big((i+p) \partial +(i+j+2p)\lambda\Big)\bar G_{i+j},
\\&&
[\bar W_i\, {}_\lambda \, \bar G_j]=\bar G_{i+j},\ [\bar W_i\, {}_\lambda \, \bar W_j]=[\bar G_i\, {}_\lambda \, \bar G_j]=0
\end{eqnarray*}
for $i,j\in\Z_+,p\in\C^*$ ($i+j>n$ the above relations are trivial). It follows from
that
$\mathcal{S}(p)_{[0]}\cong \mathfrak{sh},\mathfrak{s}(n)=\mathcal{S}(-n)_{[n]}$ for $n\geq1$
(also see \eqref{b3.2}, \eqref{bn3.2}).
Now we define   $\C[\partial]$-modules $\bar V_{a,b,0,d}$, $\bar V_{a,b,0,d,\sigma}$ and  $\bar V_{a,b,\sigma}$   over  $\mathfrak{s}(n)$ for $n>1$ as follows.
\begin{itemize}\parskip-7pt
\item[{\rm (1)}] $\bar V_{a,b,0,d}=\C[\partial]v$ with
 \begin{eqnarray*}
\\&&\begin{cases}
\bar L_0\,{}_\lambda\, v=-n(\partial+a\lambda+b)v,\\[4pt]
 \bar W_0\,{}_\lambda\, v=d v, \\[4pt]
\bar G_i\,{}_\lambda\,v=0, \ 0\leq i\leq n,\\[4pt]
\bar W_j\,{}_\lambda\,v=0,\ 1\leq j\leq n, \\[4pt]
\bar L_k\,{}_\lambda\,v=0,\ 1\leq k\leq n,
\end{cases}
\end{eqnarray*}
where $a,b,d\in\C$;
\item[{\rm (2)}] $\bar V_{a,b,0,d,\sigma}=\C[\partial]v_{\bar 0}\oplus\C[\partial]v_{\bar 1}$ with
 \begin{eqnarray*}
\\&&\begin{cases}
\bar L_0\,{}_\lambda\, v_{\bar0}=-n(\partial+a\lambda+b)v_{\bar0},\\[4pt]
 \bar W_0\,{}_\lambda\, v_{\bar0}=d v_{\bar0}, \\[4pt]
  \bar G_0\,{}_\lambda\, v_{\bar0}=\sigma v_{\bar1}, \\[4pt]
\bar G_i\,{}_\lambda\,v_{\bar0}=0, \ 1\leq i\leq n,\\[4pt]
\bar W_j\,{}_\lambda\,v_{\bar0}=0,\ 1\leq j\leq n, \\[4pt]
\bar L_k\,{}_\lambda\,v_{\bar0}=0,\ 1\leq k\leq n,
\end{cases}\   \mathrm{and} \
 \begin{cases}
\bar L_0\,{}_\lambda\, v_{\bar1}=-n(\partial+(a+1)\lambda+b)v_{\bar1},\\[4pt]
\bar W_0\,{}_\lambda\, v_{\bar1}=(d+1) v_{\bar1}, \\[4pt]
\bar G_i\,{}_\lambda\,v_{\bar1}=0,\ 0\leq i\leq n,\\[4pt]
\bar W_j\,{}_\lambda\,v_{\bar1}=0,\ 1\leq j\leq n,\\[4pt]
\bar L_k\,{}_\lambda\,v_{\bar1}=0,\ 1\leq k\leq n,
\end{cases}\\
\end{eqnarray*}
where $a,b,d\in\C,\sigma\in\C^*$;
 \item[{\rm (3)}] $\bar V_{a,b,\sigma}=\C[\partial]v_{\bar 0}\oplus\C[\partial]v_{\bar 1}$ with
  \begin{eqnarray*}
\\&&\begin{cases}
\bar L_0\,{}_\lambda\, v_{\bar0}=-n(\partial+a\lambda+b)v_{\bar0},\\[4pt]
 \bar W_0\,{}_\lambda\, v_{\bar0}=(a-1) v_{\bar0}, \\[4pt]
 \bar  G_0\,{}_\lambda\, v_{\bar0}=\sigma\big(\partial+a\lambda+b\big) v_{\bar1}, \\[4pt]
\bar G_i\,{}_\lambda\,v_{\bar0}=0,\ 1\leq i\leq n,  \\[4pt]
\bar W_j\,{}_\lambda\,v_{\bar0}=0,\ 1\leq j\leq n, \\[4pt]
\bar L_k\,{}_\lambda\,v_{\bar0}=0,\ 1\leq k\leq n,
\end{cases}\  \mathrm{and} \
 \begin{cases}
\bar L_0\,{}_\lambda\, v_{\bar1}=-n(\partial+a\lambda+b)v_{\bar1},\\[4pt]
 \bar W_0\,{}_\lambda\, v_{\bar1}=a v_{\bar1}, \\[4pt]
\bar G_i\,{}_\lambda\,v_{\bar1}=0, \ 0\leq i\leq n,\\[4pt]
\bar W_j\,{}_\lambda\,v_{\bar1}=0, \ 1\leq j\leq n, \\[4pt]
\bar L_k\,{}_\lambda\,v_{\bar1}=0,\ 1\leq k\leq n,
\end{cases}\\
\end{eqnarray*}
where $a,b\in\C,\sigma\in\C^*$.
\end{itemize}
The    $\C[\partial]$-modules  $\bar{\bar V}_{a,b,c,d}$, $\bar{\bar V}_{a,b,c,d,\sigma}$ and $\bar{\bar V}_{a,b,\sigma}$     over  $\mathfrak{s}(1)$ are given.
\begin{itemize}\parskip-7pt
\item[{\rm (1)}] $\bar{\bar V}_{a,b,c,d}=\C[\partial]v$ with
 \begin{eqnarray*}
\\&&\begin{cases}
L_0\,{}_\lambda\, v=-(\partial+a\lambda+b)v,\\[4pt]
 L_1\,{}_\lambda\,v=c v,\\[4pt]
 W_0\,{}_\lambda\, v=d v, \\[4pt]
  G_i\,{}_\lambda\, v=0,  \ 0\leq i\leq 1, \\[4pt]
W_1\,{}_\lambda\,v=0, \\[4pt]
\end{cases}
\end{eqnarray*}
where $a,b,c,d\in\C$;
 \item[{\rm (2)}] $\bar{\bar V}_{a,b,c,d,\sigma}=\C[\partial]v_{\bar 0}\oplus\C[\partial]v_{\bar 1}$ with
 \begin{eqnarray*}
\\&&\begin{cases}
L_0\,{}_\lambda\, v_{\bar0}=-(\partial+a\lambda+b)v_{\bar0},\\[4pt]
 L_1\,{}_\lambda\,v_{\bar0}=c v_{\bar0},\\[4pt]
 W_0\,{}_\lambda\, v_{\bar0}=d v_{\bar0}, \\[4pt]
  G_0\,{}_\lambda\, v_{\bar0}=\sigma v_{\bar1}, \\[4pt]
G_1\,{}_\lambda\,v_{\bar0}=W_1\,{}_\lambda\,v_{\bar0}=0, \\[4pt]
\end{cases}\   \mathrm{and} \
 \begin{cases}
L_0\,{}_\lambda\, v_{\bar1}=-(\partial+(a+1)\lambda+b)v_{\bar1},\\[4pt]
 L_1\,{}_\lambda\,v_{\bar1}=c v_{\bar1},\\[4pt]
 W_0\,{}_\lambda\, v_{\bar1}=(d+1) v_{\bar1}, \\[4pt]
G_i\,{}_\lambda\,v_{\bar1}=0,\ 0\leq i\leq 1,\\[4pt]
W_1\,{}_\lambda\,v_{\bar1}=0,
\end{cases}\\
\end{eqnarray*}
where $a,b,c,d\in\C,\sigma\in\C^*$;
\item[{\rm (3)}] $\bar{\bar V}_{a,b,\sigma}=\C[\partial]v_{\bar 0}\oplus\C[\partial]v_{\bar 1}$ with
  \begin{eqnarray*}
\\&&\begin{cases}
L_0\,{}_\lambda\, v_{\bar0}=-(\partial+a\lambda+b)v_{\bar0},\\[4pt]
 W_0\,{}_\lambda\, v_{\bar0}=(a-1) v_{\bar0}, \\[4pt]
  G_0\,{}_\lambda\, v_{\bar0}=\sigma\big(\partial+a\lambda+b\big) v_{\bar1}, \\[4pt]
G_1\,{}_\lambda\,v_{\bar0}=
W_1\,{}_\lambda\,v_{\bar0}=
L_1\,{}_\lambda\,v_{\bar0}=0,
\end{cases}\  \mathrm{and} \
 \begin{cases}
L_0\,{}_\lambda\, v_{\bar1}=-(\partial+a\lambda+b)v_{\bar1},\\[4pt]
 W_0\,{}_\lambda\, v_{\bar1}=a v_{\bar1}, \\[4pt]
G_i\,{}_\lambda\,v_{\bar1}=0, \ 0\leq i\leq 1,\\[4pt]
W_1\,{}_\lambda\,v_{\bar1}=L_1\,{}_\lambda\,v_{\bar1}=0,
\end{cases}\\
\end{eqnarray*}
where $a,b\in\C,\sigma\in\C^*$.
\end{itemize}
The following   $\C[\partial]$-modules $\bar{\bar{\bar V}}_{a,b,0,d}$,   $\bar{\bar{\bar V}}_{a,b,0,d,\sigma}$ and $\bar{\bar{\bar V}}_{a,b,\sigma}$   over  $\mathfrak{sh}$ are presented.
\begin{itemize}\parskip-7pt
\item[{\rm (1)}] $\bar{\bar{\bar V}}_{a,b,0,d}=\C[\partial]v$ with
 \begin{eqnarray*}
\\&&\begin{cases}
(\frac{1}{p}L_0)\,{}_\lambda\, v=(\partial+a\lambda+b)v,\\[4pt]
 W_0\,{}_\lambda\, v=d v, \\[4pt]
  G_0\,{}_\lambda\, v=0, \\[4pt]
\end{cases}
\end{eqnarray*}
where $a,b,d\in\C$;
 \item[{\rm (2)}] $\bar{\bar{\bar V}}_{a,b,0,d,\sigma}=\C[\partial]v_{\bar 0}\oplus\C[\partial]v_{\bar 1}$ with
 \begin{eqnarray*}
\\&&\begin{cases}
(\frac{1}{p}L_0)\,{}_\lambda\, v_{\bar0}=(\partial+a\lambda+b)v_{\bar0},\\[4pt]
 W_0\,{}_\lambda\, v_{\bar0}=d v_{\bar0}, \\[4pt]
  G_0\,{}_\lambda\, v_{\bar0}=\sigma v_{\bar1}, \\[4pt]
\end{cases}\   \mathrm{and} \
 \begin{cases}
(\frac{1}{p}L_0)\,{}_\lambda\, v_{\bar1}=(\partial+(a+1)\lambda+b)v_{\bar1},\\[4pt]
 W_0\,{}_\lambda\, v_{\bar1}=(d+1) v_{\bar1}, \\[4pt]
G_0\,{}_\lambda\,v_{\bar1}=0,\\[4pt]
\end{cases}\\
\end{eqnarray*}
where $a,b,d\in\C,\sigma\in\C^*$;
\item[{\rm (3)}] $\bar{\bar{\bar V}}_{a,b,\sigma}=\C[\partial]v_{\bar 0}\oplus\C[\partial]v_{\bar 1}$ with
  \begin{eqnarray*}
\\&&\begin{cases}
(\frac{1}{p}L_0)\,{}_\lambda\, v_{\bar0}=(\partial+a\lambda+b)v_{\bar0},\\[4pt]
 W_0\,{}_\lambda\, v_{\bar0}=(a-1) v_{\bar0}, \\[4pt]
  G_0\,{}_\lambda\, v_{\bar0}=\sigma\big(\partial+a\lambda+b\big) v_{\bar1}, \\[4pt]
\end{cases}\  \mathrm{and} \
 \begin{cases}
(\frac{1}{p}L_0)\,{}_\lambda\, v_{\bar1}=(\partial+a\lambda+b)v_{\bar1},\\[4pt]
 W_0\,{}_\lambda\, v_{\bar1}=a v_{\bar1}, \\[4pt]
G_0\,{}_\lambda\,v_{\bar1}=0,
\end{cases}\\
\end{eqnarray*}
where $a,b\in\C,\sigma\in\C^*$.
\end{itemize}
Based on Theorem \ref{77.1}, the classification of  non-trivial  conformal modules of rank $1+1$  over  $\mathfrak{sh}$  and $\mathfrak{s}(n)$ for $n\geq1$ can be presented.
 Moreover,  the same irreducibility assertions as those     modules of  $\mathfrak{sh}$  and $\mathfrak{s}(n)$ for $n\geq1$  in Proposition \ref{pro77.111} are obtained.
 We see that the irreducible modules of these modules exhaust all non-trivial finite irreducible conformal modules respectively  over  $\mathfrak{sh}$  and $\mathfrak{s}(n)$ for $n\geq1$.
\begin{coro}
Assume that $\bar V$ is a finite non-trivial irreducible conformal module   over  $\mathfrak{sh}$  or $\mathfrak{s}(n)$ for $n\geq1$.
\begin{itemize}\parskip-7pt
 \item[{\rm (i)}]
  If $\bar V$ is an $\mathfrak{s}(n)$-module for $n>1$, then
   \begin{eqnarray*}
\bar V\cong
\begin{cases}
\bar V_{a,b,0,d} &\ \mbox{for}\  (a,d)\neq(0,0),b\in\C,\\[4pt]
\bar V_{a,b,0,d,\sigma} &\ \mbox{for}\  (a,d)\neq(0,0),b\in\C,\sigma\in\C^*,\\[4pt]
\bar V_{a,b,\sigma}&\  \mbox{for} \ a,\sigma\in\C^*,b\in\C,\\[4pt]
\Pi({\bar V}_{a-1,b,0,d-1,\sigma})&\  \mbox{for} \ (a,d)\neq(1,1),b\in\C,\sigma\in\C^*,\\[4pt]
\Pi({\bar V}_{a,b,\sigma})&\  \mbox{for} \ a,\sigma\in\C^*,b\in\C;
\end{cases}\\
\end{eqnarray*}
 \item[{\rm (ii)}]
  If $\bar V$ is an $\mathfrak{s}(1)$-module, then
     \begin{eqnarray*}
\bar V\cong
\begin{cases}
\bar{\bar{V}}_{a,b,c,d} &\ \mbox{for}\  (a,c,d)\neq(0,0,0),b\in\C,\\[4pt]
\bar{\bar{V}}_{a,b,c,d,\sigma} &\ \mbox{for}\  (a,c,d)\neq(0,0,0),\sigma\in\C^*,b\in\C,\\[4pt]
\bar{\bar{V}}_{a,b,\sigma}&\  \mbox{for} \ a,\sigma\in\C^*,b\in\C,\\[4pt]
\Pi(\bar{\bar V}_{a-1,b,c,d-1,\sigma})&\  \mbox{for} \ (a,c,d)\neq(1,0,1),\sigma\in\C^*,b\in\C,\\[4pt]
\Pi(\bar{\bar V}_{a,b,\sigma})&\  \mbox{for} \ a,\sigma\in\C^*,b\in\C;
\end{cases}\\
\end{eqnarray*}
  \item[{\rm (iii)}]
  If $\bar V$ is an $\mathfrak{sh}$-module, then
        \begin{eqnarray*}
\bar V\cong
\begin{cases}
\bar{\bar{\bar{V}}}_{a,b,0,d} &\ \mbox{for}\  (a,d)\neq(0,0),b\in\C,\\[4pt]
\bar{\bar{\bar{V}}}_{a,b,0,d,\sigma} &\ \mbox{for}\  (a,d)\neq(0,0),\sigma\in\C^*,b\in\C,\\[4pt]
\bar{\bar{\bar{V}}}_{a,b,\sigma}&\  \mbox{for} \ a,\sigma\in\C^*,b\in\C,\\[4pt]
\Pi(\bar{\bar{\bar V}}_{a-1,b,0,d-1,\sigma})&\  \mbox{for} \ (a,d)\neq(1,1),\sigma\in\C^*,b\in\C,\\[4pt]
\Pi(\bar{\bar{\bar V}}_{a,b,\sigma})&\  \mbox{for} \ a,\sigma\in\C^*,b\in\C.
\end{cases}\\
\end{eqnarray*}
\end{itemize}
\end{coro}

\begin{rema}
By \eqref{ln23}, we also obtain the classification of     all non-trivial finite irreducible conformal modules over a subalgebra of Lie conformal algebra of  $N=2$ superconformal algebra.
\end{rema}

\section*{Acknowledgements}
This work was supported by the National Natural Science Foundation of China (No.11801369,   11871421,  11971350), the Zhejiang Provincial Natural Science Foundation of China (No. LY20A010022)
and the Scientific Research Foundation of Hangzhou Normal University (No. 2019QDL012). We also thank the referee for nice suggestions to make the paper more readable.
\section*{Data Availability} Data sharing no applicable to this article.

\bigskip

Haibo Chen
\vspace{2pt}

  School of  Statistics and Mathematics, Shanghai Lixin University of  Accounting and Finance,   Shanghai
201209, China

\vspace{2pt}
rebel1025@126.com

\bigskip

Yanyong Hong
\vspace{2pt}

Department of Mathematics, Hangzhou Normal University,
Hangzhou 311121,  China

\vspace{2pt}
hongyanyong2008@yahoo.com

\bigskip

Yucai Su
\vspace{2pt}

  School of Mathematical Sciences, Tongji University, Shanghai
200092, China
\vspace{2pt}

ycsu@tongji.edu.cn

\small

\begin{thebibliography}{9999}\vskip0pt\small
\def\re{\bibitem}\parindent=2ex\parskip=-2pt\baselineskip=-2pt

\bibitem{BKV} B. Bakalov, V. G.  Kac, A. A. Voronov, Cohomology of conformal algebras, {\it Comm. Math. Phys.} {\bf 200}, 561--598 (1999).





\bibitem{BKL1}
C. Boyallian,   V. G.  Kac,   J. I. Liberati,
Classification of finite irreducible modules over the Lie conformal superalgebra $CK_6$,
{\it Comm. Math. Phys.} {\bf 317},   503--546 (2013).

\bibitem{BKL2}
C. Boyallian,   V. G.  Kac,   J. I. Liberati,
Irreducible modules over finite simple Lie conformal superalgebras of type $K$,
{\it J. Math. Phys.} {\bf 51},  063507 (2010).



\bibitem{BKL3}
C. Boyallian,   V. G.  Kac,   J. I. Liberati, A. Rudakov,
Representations of simple finite Lie conformal superalgebras of type $W$ and $S$,
{\it J. Math. Phys.} {\bf 47},    043513 (2006).






\bibitem{CHSX} H. Chen,   J. Han, Y. Su, Y. Xu, Loop Schr\"{o}dinger-Virasoro Lie conformal algebra,  {\it Internat. J. Math.} {\bf 27}, 1650057 (2016).


\bibitem{CHS} H. Chen,   Y. Hong, Y. Su, Finite irreducible conformal modules over the extended Block
type Lie conformal algebra $\mathfrak{B}(\alpha,\beta,p)$, arXiv:1908.05827.


\bibitem{CK} S. Cheng, V. G. Kac, Conformal modules, {\it Asian J. Math.} {\bf 1}, 181--193 (1997),
{\it Asian J. Math.}  {\bf 2}, 153--156  (1998)  (Erratum).



\bibitem{CK1} S. Cheng, V. G.  Kac,   M. Wakimoto, Extensions of conformal modules, in: Topological Field Theory, Primitive Forms and Related Topics(Kyoto), in: {\it Progress in Math.} {\bf 160}, Birkh$\ddot{a}$user, Boston, 33--57; q-alg/9709019 (1998).


\bibitem{CL}
S. Cheng, N. Lam,
Finite conformal modules over $N=2,3,4$
superconformal algebras,
{\it J. Math. Phys.} {\bf 42}, 906--933  (2001).

\bibitem{DH}
X. Dai, J. Han, Loop Virasoro Lie conformal superalgebra,
{\it Linear Multilinear Algebra}
{\bf  66}, 131--146 (2018).


\bibitem{DK} A. D'Andrea, V. G. Kac, Structure theory of finite conformal algebras, {\it Sel. Math.} {\bf 4}, 377--418 (1998).


\bibitem{FK}
 D. Fattori,  V. G. Kac,  Classification of finite simple Lie conformal superalgebras, {\it J. Algebra } {\bf  258},
23--59 (2002).

\bibitem{DVV}
R. Dijkgraaf, E. Verlinde, H. Verlinde, Topological strings in $d<1$, {\it Nucl. Phys. B} {\bf 352},
 59--86 (1991).

\bibitem{FSW} G. Fan, Y. Su, H. Wu, Loop Heisenberg-Virasoro Lie conformal algebra, {\it J. Math. Phys.}
{\bf 55}, 123508 (2014).

\bibitem{HW} Y. Hong, Z. Wu, Simplicity of quadratic Lie conformal algebras, {\it Comm. Algebra} {\bf 45}, 141--150 (2017).


\bibitem{K3} V. G.  Kac, {\it Vertex algebras for beginners}, University Lecture Series Vol. {\bf10}, American Mathematical Society, 1996.

\bibitem{K1} V. G.  Kac, Formal distribution algebras and conformal algebras, XIIth International Congress of Mathematical Physics (ICMP'97) (Brisbane), 80-97, Int. Press, Cambridge, MA, 1999.

\bibitem{K2} V. G.  Kac, The idea of locality, in {\it Physical Applications and Mathematical Aspects
of Geometry, Groups and Algebras,}  World Scientific, Singapore,
  16-32 (1997).


\bibitem{LHW} L. Luo,  Y.   Hong,  Z.  Wu,
Finite irreducible modules of Lie conformal algebras
$\mathcal{W}(a,b)$ and some Schr\"{o}dinger-Virasoro type
Lie conformal algebras, {\it Int. J. Math.} {\bf 30},  1950026 (2019).
\bibitem{M}
M. Mansour, On the quantum super Virasoro algebra, {\it  Czech. J. Phys.} {\bf 51}, 883--888 (2001).

\bibitem{MZ}
C. Martinez, E. Zelmanov,
Irreducible representations of the exceptional Cheng-Kac superalgebra
{\it Trans. Amer. Math. Soc.} {\bf 366},  5853--5876  (2014).


\bibitem{SXY}  Y.  Su, C. Xia, L. Yuan,  Classification of finite irreducible conformal modules over a class of Lie conformal algebras of Block type,
{\it J. Algebra} {\bf 499},  321--336 (2019).

\bibitem{SY2} Y.  Su,   X. Yue, Filtered Lie conformal algebras whose associated graded algebras are isomorphic to that of general
conformal algebra $gc_1$, {\it J. Algebra} {\bf 340}, 182--198 (2011).





\bibitem{WCY} H.  Wu, Q. Chen,  X. Yue, Loop Virasoro Lie conformal algbera, {\it J. Math. Phys.} {\bf 55}, 011706 (2014).

\bibitem{WY} H. Wu, L. Yuan, Classification of finite irreducible conformal modules over some
Lie conformal algebras related to the Virasoro conformal algebra, {\it J. Math. Phys.}  {\bf 58},
 041701 (2017).


\bibitem{X} C. Xia,  Classification of finite irreducible conformal modules over Lie conformal superalgebras of Block type,  {\it J. Algebra} {\bf 531}, 141--164  (2019).

\end{thebibliography}
\end{document}